\documentclass[12pt,a4paper,twoside]{article}

\usepackage[applemac]{inputenc}
\usepackage[T1]{fontenc}
\usepackage[english]{babel}
\usepackage{fancyhdr}
\usepackage{indentfirst} 
\usepackage{newlfont}
\usepackage{verbatim}
\usepackage{enumerate}
\usepackage{hyperref}
\usepackage{amsthm}  
\usepackage{amssymb}  
\usepackage{amscd}
\usepackage{amsmath}
\usepackage{latexsym} 
\usepackage{amsfonts}
\usepackage{amsbsy}   
\usepackage{mathrsfs} 
\usepackage{arcs}
\usepackage{yhmath}
\usepackage{graphicx}
\usepackage{color}
\usepackage{subfigure}
\usepackage{lscape}

\theoremstyle{plain}
\newtheorem{teo}{Theorem}[section]

\newtheorem*{teo**}{Theorem \mynumber}
\newenvironment{teo*}[1]
  {\newcommand{\mynumber}{\ref{#1}}\begin{teo**}}
  {\end{teo**}}

\newtheorem*{thm*}{Theorem}
\newtheorem*{clm*}{Claim}
\newtheorem{prop}{Proposition}[section]

\newtheorem{cor}{Corollary}[section]
\newtheorem{lem}{Lemma}[section]

{\Alph{app}}

\makeatletter
\newtheorem*{rep@theorem}{\rep@title}
\newcommand{\newreptheorem}[2]{%
\newenvironment{rep#1}[1]{%
 \def\rep@title{#2 \ref{##1}}%
 \begin{rep@theorem}}%
 {\end{rep@theorem}}}
\makeatother

\newreptheorem{propo}{Proposition}

\newreptheorem{teoo}{Theorem}

\theoremstyle{definition}
\newtheorem{defin}{Definition}[section]
\newtheorem{ese}{Example}[section]

\theoremstyle{remark}
\newtheorem{oss}[teo]{Remark}

\begin{document}

\title{An alternative proof of infinite dimensional Gromov's non-squeezing for compact perturbations of linear maps}

\author{Lorenzo Rigolli\footnote{This work was partially supported by the {DFG grant AB 360/1-1.}  
 \qquad \qquad   \qquad \qquad E-mail:\textit{lorenzo.rigolli@rub.de}.}}
\maketitle

\begin{abstract}
This paper deals with the problem of generalising Gromov's non squeezing theorem to an infinite dimensional Hilbert phase space setting. 
By following the lines of the proof by Hofer and Zehnder of finite dimensional non-squeezing, we recover an infinite dimensional non-squeezing result by Kuksin \cite{Kuk95a} for symplectic diffeomorphisms which are non-linear compact perturbations of a symplectic linear map.
We also show that the infinite dimensional non-squeezing problem, in full generality, can be reformulated as the problem of finding a suitable Palais-Smale sequence for a distinguished Hamiltonian action functional.
\end{abstract}
\ \\ 

\section*{Introduction}
A \textit{strong symplectic form} on a real Hilbert space $\mathbb{H}$ is a skew-symmetric continuous bilinear form
\begin{align*}
\omega : \mathbb{H} \times \mathbb{H} \rightarrow \mathbb{R},
\end{align*}
which is strongly non-degenerate, namely the associated bounded linear operator $\Omega : \mathbb{H} \rightarrow \mathbb{H}^*$ defined by
\begin{align*}
< \Omega x , y >= \omega (x,y) \ \ \forall x,y \in \mathbb{H},
\end{align*}
where $< \cdot , \cdot >$ denotes the duality pairing, is an isomorphism. A Hilbert space $(\mathbb{H},\omega)$ endowed with a strong symplectic form $\omega$ is said \textit{symplectic Hilbert space}.
In this paper we study the global 2-dimensional rigidity phenomenon given by Gromov's non-squeezing, in the setting of infinite dimensional symplectic Hilbert spaces.\\ 
As first noticed by Zakharov \cite{Zak74}, the Hamiltonian formalism is useful not only to study classical mechanics but also some evolutionary PDEs that satisfy the conservation of energy.
The solutions of these PDEs are given by the Hamiltonian flow associated to a scalar function defined on a distinguished infinite dimensional symplectic space, which is usually uniquely determined by the PDE under consideration.\\
Just to make an example, under suitable assumptions on the nonlinearity $f$, the nonlinear wave equation:
\begin{align*}
u_{tt} - \Delta u + f(u)=0,
\end{align*}
for $u : \mathbb{R} \times \mathbb{T}^n \rightarrow \mathbb{R}$, where $\mathbb{T}^n$ is the $n$-torus, defines a Hamiltonian flow on the Sobolev space $H^{\frac{1}{2}} (\mathbb{T}^n) \times H^{\frac{1}{2}} (\mathbb{T}^n)$ which carries a strong symplectic form. 
For infinite dimensional Hamiltonian systems like this it makes sense to speculate about the validity of non-squeezing.\\  \\
\textbf{(Infinite dimensional squeezing question)} Let $(\mathbb{H},\omega)$ be a symplectic Hilbert space with a compatible inner product.  For $0 < s < r$, is it possible for a Hamiltonian flow to map a ball of radius $r$ in $\mathbb{H}$ into a cylinder given by all vectors of $\mathbb{H}$ whose symplectic projection onto a plane $V$ has distance less than $s$ from a given vector in $V$?\\ \\
As first observed by Kuksin \cite{Kuk95a}, the non-squeezing property for nonlinear Hamiltonian PDEs would have relevant physical implications; for instance in the \textit{transfer of energy problem} of understanding whether the energy of nonlinear conservative oscillations spreads to higher frequencies.  
Moreover non-squeezing would also imply \textit{uniformly asymptotically stability} for solutions of Hamiltonian PDEs.\\
Prompted by these motivations Kuksin showed that the non-squeezing holds if the flow of the PDE is a compact perturbation of a linear one and some technical assumptions are fulfilled \cite{Kuk95a}.
Not all Hamiltonian PDEs have a flow of this form, however under certain conditions this is the case for nonlinear hyperbolic equations like the nonlinear versions of the wave equation, the Schr\"odinger equation, the membrane equation and the string equation.\\
The non-squeezing property was later proved also for some equations whose flow is not a compact perturbation of a linear one, like the cubic Schr\"odinger equation considered by Bourgain \cite{Bou94b}, the KdV equation considered by Colliander, Keel, Staffilani, Takaoka and Tao \cite{CKS+05}, and the BBM equation considered by Roum\'egoux \cite{Rou10}.
This was done using quite delicate approximations of the flow with a finite dimensional one and then applying standard Gromov's non-squeezing.\\
At this point we remark that two different notions of symplectic form are commonly used in the Hilbert space setting, namely \textit{strong} symplectic forms and \textit{weak} symplectic forms.
Strong symplectic forms are usually defined on Sobolev spaces of functions with low regularity and in this case it makes sense to investigate the validity of non-squeezing; anyway it is important to notice that if the phase space associated to the PDE is a Hilbert space of smoother functions which carry only a so-called \textit{weak} symplectic form (namely a closed non-degenerate 2-form), then the squeezing is possible and energy transfer is typical, see for example \cite{Kuk95b} and \cite{CKS+10}.\\
A natural approach to deal with the infinite dimensional squeezing question is extending well known symplectic capacities to infinite dimensional subsets, since the existence of a normalized symplectic capacity implies the non-squeezing.
In this direction, Abbondandolo and Majer \cite{AM15} used the convex geometry tool given by Clarke's duality in order to construct an infinite dimensional symplectic capacity for convex sets and thus to prove a non-squeezing theorem when the image of the unit ball under a symplectomorphism is a convex set.\\

The main goal of this paper is to prove an infinite dimensional non-squeezing result which slightly generalizes the result by Kuksin in \cite{Kuk95a}.
In order to give a precise statement we briefly recall some basic concepts about symplectic structures on infinite dimensional separable Hilbert spaces (for more details see Chapter 2 of \cite{AM15} or \cite{CM74}, in which a more general Banach setting is considered).\\
Let $(\mathbb{H},\omega)$ be a symplectic Hilbert space; the choice of a Hilbert inner product $\langle \cdot, \cdot \rangle$ on $\mathbb{H}$ determines a bounded linear operator
$J : \mathbb{H} \rightarrow \mathbb{H}$ such that
\begin{align*}
\langle Jx,y \rangle = \omega(x,y) \ \ \forall x,y \in \mathbb{H}.
\end{align*}
Being the composition of $\Omega$ by the isomorphism $\mathbb{H}^* \cong \mathbb{H}$ induced by the inner product, $J$ is also an isomorphism. The skew symmetry of $\omega$ now reads as $J^T = -J$, where $J^T : \mathbb{H} \rightarrow \mathbb{H}$ is the adjoint operator with respect to the inner product. \\
We say that an inner product $\langle \cdot , \cdot \rangle$ is \textit{compatible} with $\omega$ if one of the following conditions (which are actually equivalent) holds
\begin{enumerate}[1)]
\item $\Omega$ is an isometry (where $\mathbb{H}^*$ is endowed with the dual norm),
\item $J$ is an isometry,
\item $J$ is a complex structure (i.e. $J^2 = -I$).
\end{enumerate}
One can show that every symplectic Hilbert space $(\mathbb{H},\omega)$ admits a compatible inner product $\langle \cdot , \cdot \rangle$. In general there is not a unique inner product compatible with a fixed $\omega$, nevertheless the unit balls corresponding to different compatible inner products are all linearly symplectomorphic.\\

Let $(\mathbb{H},\omega)$ be a symplectic Hilbert space endowed with a compatible inner product $\langle \cdot, \cdot \rangle$ and let $\{e_i,f_i\}_{i \in \mathbb{N}}$ be a countable orthonormal basis such that any $\{e_i,f_i\}$ spans a symplectic plane. For any $n \in \mathbb{N}$ we consider the orthogonal projections
\begin{align}
\label{proiez}
\begin{split}
P_n :\mathbb{H} &\rightarrow \mathbb{H}_n\\
 x& \mapsto \sum_{i=1}^{n} \langle x,e_i \rangle e_i +  \langle x,f_i \rangle f_i,
 \end{split}
\end{align}  
onto the $2n$-dimensional symplectic Hilbert subspace $\mathbb{H}_n$.\\
Let us consider the group of admissible symplectomorphisms
\begin{align*}
Symp_a (\mathbb{H},\omega,\langle \cdot, \cdot \rangle):=& \big{\{} \varphi \in Symp(\mathbb{H},\omega) \ \big{|}  \textrm{ for } k=\pm 1, \ D \varphi^{k} \textrm{ and } D^2 \varphi^{k} \textrm{ are bounded}; \\
& (I-P_n) {\varphi^k}_{\vert P_n \mathbb{H}} \underset{n \rightarrow + \infty} \longrightarrow 0 \textrm{ uniformly on bounded sets}; \\
& [P_n , D \varphi^k (x)^*]  \underset{n \rightarrow + \infty} \longrightarrow 0 \textrm{ in operators' norm, uniformly in} \ x \in \mathbb{H} \\
&\textrm{on bounded sets} \big{\}}.
\end{align*}
We will prove the following infinite dimensional non-squeezing result.
\begin{thm*}[Infinite dimensional non-squeezing]
Let $(\mathbb{H},\omega)$ be a Hilbert symplectic space endowed with a compatible inner product $\langle \cdot, \cdot \rangle$, $B_r$ the ball centred in $0$ with radius $r$ and $Z_R$ a cylinder whose basis lays on a symplectic plane and has symplectic area $\pi R^2$. Let $\varphi \in Symp_a (\mathbb{H},\omega,\langle \cdot, \cdot \rangle)$, if $\varphi(B_r) \subset Z_R$ then $r\leq R$.
\end{thm*}
We will see that the theorem above applies to symplectomorphisms which are compact perturbations of linear maps, under slightly less restrictive assumptions than the ones considered in \cite{Kuk95a}.\\ 
In view of applications we remark that the condition we require on the boundedness of the differentials of $\varphi \in Symp_a (\mathbb{H},\omega, \langle \cdot , \cdot \rangle)$ is not very restrictive, in fact the differentials of any order of the flow of a typical Hamiltonian PDE which is well-posed on its phase space are bounded on bounded sets.\\ 
To prove the theorem we follow the approach adopted by Hofer and Zehnder in order to deduce the non-squeezing theorem for symplectomorphisms of $(\mathbb{R}^{2n},\omega_0)$, see \cite{HZ94}.
What they did is observing that the non-squeezing is implied by the existence of an appropriate critical point for a distinct Hamiltonian action functional and then they proved the existence of such a critical point by applying a minimax argument to the action functional.\\
The main difference in dealing with infinite dimensional spaces instead of with Euclidean spaces is that we face a loss of local compactness: this has two major manifestations.\\
First, unlike in the finite dimensional case, the Hamiltonian action functional we are interested in does not necessarily satisfy the Palais-Smale condition and so we may find (PS) sequences which do not converge to a critical point.
Therefore, in the infinite dimensional setting, (PS) sequences of the action functional play the role which in the Hofer-Zehnder setting is played by critical points.
Secondly, in the infinite dimensional setting, in order to be able to determine a (PS) sequence for the action functional we may have to ask for an additional compactness property.
This compactness can be recovered by restricting the class of symplectomorphisms under consideration to the one of admissible symplectomorphisms.\\
We work with admissible symplectomorphisms only in order to find suitable Palais-Smale sequences, but as we will show, the fact that the infinite dimensional non-squeezing question can be reduced to a purely critical point theory problem holds without imposing any restriction to the set of symplectomorphisms under consideration.\\ \\
\\ \\ 
The paper is organized as follows.\\
In Section \ref{sec1} we introduce the analytical setting in which we will work.\\
In Section \ref{sec2} we characterize (PS) sequences of the Hamiltonian action functional and we show that they are preserved by symplectomorphisms as well as their corresponding almost critical levels.\\
In Section \ref{sec3} we prove the infinite dimensional non-squeezing theorem via a minimax argument in combination with a finite dimensional approximation technique and we notice that this generalizes the non-squeezing theorem by Kuksin.\\
In Section \ref{sec4} we show how, in bigger generality, the infinite dimensional non-squeezing problem can be formulated as a critical point theory problem.\\

{\textbf{Acknowledgments.}}
I would like to warmly thank Alberto Abbondandolo for all the precious help and advice he gave me concerning this paper.
\section{Analytical setting}
\label{sec1}
Let $(\mathbb{H},\omega)$ be a separable symplectic Hilbert space endowed with a compatible inner product. On a countable orthonormal basis $\{e_i,f_i \}_{i \in \mathbb{N}}$ the complex structure is defined by
\begin{align*}
J(e_i) &=f_i,\\
J(f_i) &=-e_i,
\end{align*}
for any $i$.
Given any smooth time-independent Hamiltonian function $H \in  C^{\infty}(\mathbb{H},\mathbb{R})$, the associated Hamiltonian action functional $\mathcal{A}_H: C^\infty (S^1,\mathbb{H}) \rightarrow \mathbb{R}$ can be written as
\begin{align*}
\mathcal{A}_H(x)= -\dfrac{1}{2} \int_0 ^1 \langle J \dot{x}(t) ,x(t) \rangle dt - \int_0 ^1 H(x(t)) dt,
\end{align*}
where $C^\infty (S^1,\mathbb{H})$ is the space of smooth loops with values in $\mathbb{H}$.\\
For our purposes it is more convenient to deal with a Hilbert space than with the space of smooth loops, thus we observe that any loop $x\in C^\infty (S^1,\mathbb{H})$ is an element of  $L^2(S^1,\mathbb{H})$ which can be represented by its Fourier series as
\begin{align*}
x(t)= \sum_{k \in \mathbb{Z}} e^{2 \pi k J t} x^k,
\end{align*}
with coefficients $x^k \in \mathbb{H}$ which are defined by the formula
\begin{align*}
x^k := \int_0^1 e ^{-2\pi kt J } x(t)  dt, \ \ k \in \mathbb{Z}.
\end{align*}
The first term of the action functional describes its symplectic part
\begin{align*}
a(x,y):=-\dfrac{1}{2} \int_0 ^1 \langle J \dot{x}(t) ,y(t) \rangle dt,
\end{align*}
while the second one describes the part depending on the Hamiltonian
\begin{align*}
b_H(x)= \int_0 ^1 H(x(t)) dt.
\end{align*}
Inserting the Fourier expansion of $x,y \in C^\infty (S^1,\mathbb{H})$ into $a(x,y)$ and noticing that 
\begin{align*}
\int_0^1 \langle e^{2 \pi j J t} x^j, e^{2 \pi k J t} y^k\rangle dt = \delta_{jk} \langle x^j , y^k \rangle,
\end{align*}
we obtain that
\begin{align*}
2a(x,y)= 2 \pi  \sum_{j \in \mathbb{Z}} j \langle x^j , y^j \rangle =2 \pi  \sum_{j >0} |j| \langle x^j , y^j \rangle  - 2 \pi  \sum_{j <0} |j| \langle x^j , y^j \rangle.
\end{align*}
Let us recall that for $s$ non-negative, the fractional Sobolev spaces are defined as
 \begin{align*}
H^{s}(S^1,\mathbb{H}):=\{ x \in L^2 (S^1, \mathbb{H}) \ \big{|} \sum_{k\in \mathbb{Z}} ^\infty |k|^{2s} \|x^k \|^2 < \infty \}.
\end{align*}
The standard inner product on the Hilbert space $H^{s}(S^1,\mathbb{H})$ is given by
\begin{align*}
\langle x,y\rangle_{s}:=\langle x^0,y^0\rangle + 2 \pi \sum_{k\neq 0} |k|^{2s} \langle x^k ,y^k\rangle,
\end{align*} 
and the induced norm is
\begin{align*}
| x |_{s}:=\sqrt{\langle x,x\rangle_s}.
\end{align*} 
Therefore, by the previous remark the functional $a$ can be extended to the space 
\begin{align*}
H^{\frac{1}{2}}(S^1,\mathbb{H}):=\{ x \in L^2 (S^1, \mathbb{H}) \ \big{|} \sum_{k \in \mathbb{Z}} ^\infty |k| \|x^k \|^2 < \infty \},
\end{align*}
which is strictly larger than the space of smooth loops. Its elements are loops (not necessarily continuous or bounded) taking values in the Hilbert symplectic space $\mathbb{H}$.\\
On $H^{\frac{1}{2}}(S^1,\mathbb{H})$ it is possible to define another equivalent norm.
The Slobodeckij semi-norm of an element $x \in H^s(S^1,\mathbb{H})$ with $0<s<1$ is the quantity
\begin{align*}
| x | _{sem,s} := \big(\int_{S^1} \int_{S^1} \dfrac{|x(t)-x(r)|^2}{|t-r|^{1+2s}}dtdr\big)^{\frac{1}{2}}.
\end{align*}
This semi-norm induces a norm on $H^{s}(S^1,\mathbb{H})$
\begin{align*}
| x |' _s :=  | x | _{L^2} + | x | _{sem,s},
\end{align*}
such that there exist constants $C_1,C_2 \geq 0$ for which
\begin{align*}
C_1| x |' _s \leq | x | _s \leq C_2 | x |' _s,
\end{align*}
for any $x \in H^s(S^1,\mathbb{H})$.\\
To keep the notation shorter we denote
\begin{align*}
E:=H^{\frac{1}{2}}(S^1,\mathbb{H}).
\end{align*}
Let us consider the orthogonal spitting 
\begin{align*}
E=E^- \oplus E^0 \oplus E^+,
\end{align*}
into the spaces having non vanishing Fourier coefficients only for $k<0$, $k=0$ and $k>0$ respectively.\\
If we denote by $P^-$, $P^0$ and $P^+$ the orthogonal projections onto the linear subspaces $E^-$, $E^0$, $E^+$, we get that any element of $E$ can be written as 
\begin{align*}
x= x^- + x^0 + x^+,
\end{align*}
where $P^-(x)=x^-$,  $P^0(x)=x^0$ and  $P^+(x)=x^+$.\\
Summarizing the previous discussion, we saw that the map
\begin{align*}
a(x,y)= \dfrac{1}{2} \langle (P^+ -P^-)x , y \rangle_{\frac{1}{2}} =\dfrac{1}{2} \langle x^+ , y^+ \rangle_{\frac{1}{2}} - \dfrac{1}{2} \langle x^- , y^- \rangle_{\frac{1}{2}}
\end{align*}
is a continuous bilinear form on $E$, whose associated quadratic form
\begin{align*}
a(x)= a(x,x)= \dfrac{1}{2} |x^+|_{\frac{1}{2}}^2 - \dfrac{1}{2} | x^-|_{\frac{1}{2}}^2
\end{align*}
is  negative definite on $E^-$, positive definite on $E^+$ and has $E^0$ as kernel.\\
Its differential is given by 
\begin{align*}
da(x)[y]= \langle (P^+ -P^-)x , y \rangle_{\frac{1}{2}},
\end{align*}
therefore its gradient with respect to the inner product on $E$ is
\begin{align*}
\nabla_{\frac{1}{2}}a(x)= (P^+ -P^-)x.
\end{align*}
Before proceeding with the study of the action functional $\mathcal{A}_H$ it is useful to state a couple of general results regarding fractional Sobolev spaces.
\begin{prop}
The space $H^{\frac{1}{2}}(S^1,\mathbb{H})$ is continuously embedded into $L^p(S^1,\mathbb{H})$ for any $1\leq p < \infty$. More precisely there exist a constant $C_p$ such that
\begin{align*}
| x |_{L^p} \leq C_p | x |_{\frac{1}{2}},
\end{align*}
for any $H^{\frac{1}{2}}(S^1,\mathbb{H})$.
\end{prop}
\proof
Using the trace theorem the result can be reduced to an application of the Sobolev's embedding theorem for integer Sobolev spaces.
See Appendix A.3 of \cite{HZ94} for the analogous proof in the case of loops taking values in $\mathbb{R}^{2n}$.
\endproof
\begin{lem}
\label{lemgrad}
Given any $s\geq 0$, let us consider the continuous inclusion
\begin{align*}
T: H^{s}(S^1,\mathbb{H}) \rightarrow L^2(S^1,\mathbb{H}).
\end{align*}
The adjoint operator $T^*$ can be represented as
\begin{align}
\label{vvv}
T^* (y)=y^0 + \sum_{k\neq 0} \dfrac{1}{2 \pi |k|^{2s}}e^{2 \pi k J t} y^k,
\end{align}
and if $f:H^{s}(S^1,\mathbb{H}) \rightarrow \mathbb{R}$ is a differentiable map, then the following equality holds
\begin{align}
\label{uuu}
T^* \nabla_{L^2} f(x)=\nabla_{s} f(x),
\end{align}
for any $x \in H^{s}(S^1,\mathbb{H})$.
\end{lem}
\proof
Using the definition of adjoint we get
\begin{align*}
\sum_{k\in \mathbb{Z}} \langle x^k,y^k \rangle = \langle x^0,T^*(y)^0 \rangle + 2 \pi \sum_{k \neq 0} |k|^{2s} \langle x^k,T^*(y)^k \rangle
\end{align*}
for any $x \in H^{s}$ and $y\in L^2$, thus the equality \eqref{vvv}.\\
For any $x, y \in H^{s}$ we get
\begin{align*}
 \langle  T^* \nabla_{L^2} f(x), y  \rangle_{s} = \langle  \nabla_{L^2} f(x),Ty \rangle_{L^2} = \langle \nabla_{L^2} f(x) , y \rangle_{L^2}  = df(x)[y] = \langle \nabla_{s} f(x) , y \rangle_{s}
\end{align*}
hence the equality \eqref{uuu}.
\endproof

We are now ready to focus on the properties of the Hamiltonian part of the action functional $\mathcal{A}_H$, namely
\begin{align*}
b_H(x)= \int_0 ^1 H(x(t)) dt,
\end{align*}
with $H \in C^{\infty}(\mathbb{H},\mathbb{R})$ autonomous Hamiltonian.
\begin{prop}
\label{diff}
If all the derivatives of $H \in C^{\infty}(\mathbb{H},\mathbb{R})$ have at most polynomial growth, i.e. if for every $k\in \mathbb{N}$ there are constants $C_k$ and $N_k$ such that 
\begin{align*}
\| d^k H(x) \| \leq C_k (1+ \| x\|^{N_k}) \ \ \ \ \forall x \in \mathbb{H},
\end{align*}
then the functional $b_H:E \rightarrow \mathbb{R}$ is $C^\infty$. Moreover
\begin{align*}
d^k b_H (x)[u]^k= \int _0 ^1 d^k H(x(t))[u(t)]^k dt \ \ \forall x,u \in E.
\end{align*}
\end{prop}
\proof
The estimate on the growth of $H$ and the fact that $H^{\frac{1}{2}}(S^1, \mathbb{H})$ is continuously embedded in any $L^p$ space for $p\neq \infty$ imply that $b_H$ is defined on $E$.
The Taylor expansion of $H$ allows to recover the Taylor series of $b_H$ and again using the hypothesis on $H$ and the continuity of the embedding of $H^{\frac{1}{2}}(S^1,\mathbb{H})$ into any $L^p(S^1,\mathbb{H})$ one deduces that $b_H$ has the same regularity as $H$.
\endproof

Starting from this point all the Hamiltonians $H \in C^{\infty}(\mathbb{H},\mathbb{R})$ we consider are assumed to have globally Lipschitz continuous gradient
\begin{align*}
\mathbb{H} &\rightarrow \mathbb{H}\\
x &\mapsto \nabla H(x).
\end{align*}

\begin{lem}
The map
\begin{align*}
H^{\frac{1}{2}}(S^1,\mathbb{H}) &\rightarrow H^{\frac{1}{2}}(S^1,\mathbb{H}) \\
x &\mapsto \nabla_{\frac{1}{2}} b_H(x) 
\end{align*}
is globally Lipschitz continuous.
\label{lemlip}
\end{lem}
\proof
If we consider the inclusion
\begin{align*}
T: H^{\frac{1}{2}}(S^1,\mathbb{H}) \rightarrow L^2(S^1,\mathbb{H}),
\end{align*}
equality \eqref{vvv} implies that 
\begin{align}
\label{lll}
| T^* (y)|_1 \leq | y|_{L^2}.
\end{align}
Using \eqref{uuu} we obtain
\begin{align*}
T^* \nabla H(x) =T^* \nabla_{L^2} b_H(x)=\nabla_{\frac{1}{2}} b_H(x)
\end{align*}
and by \eqref{lll} together with the assumption that $\nabla H$ is Lipschitz continuous on $\mathbb{H}$, we get that
\begin{align*}
| \nabla_{\frac{1}{2}} b_H(x) - \nabla_{\frac{1}{2}} b_H (y) |_{\frac{1}{2}} &=| T^* ( \nabla H(x) -  \nabla H(y)) |_{\frac{1}{2}}\\
&\leq | \nabla H(x) -  \nabla H(y) |_{L^2} \\
& \leq  C | x-y|_{L^2}\\
& \leq  C | x-y|_{\frac{1}{2}}
\end{align*}
and hence
\begin{align*}
| \nabla_{\frac{1}{2}} b_H(x) - \nabla_{\frac{1}{2}} b_H (y) |_{\frac{1}{2}}\leq C | x-y|_{\frac{1}{2}}.
\end{align*}
\endproof
For any $k\in \mathbb{Z}$, let $P^k:E \rightarrow E$ be the orthogonal projector onto the space corresponding to the $k$-th Fourier mode
\begin{align*}
(P^k x)(t)= e^{2 \pi k t J } x^k, \ \ x^k \in \mathbb{H}.
\end{align*}
We search for an estimate on the norm of the $k$-th Fourier coefficients $\nabla_{\frac{1}{2}} b_H(x)^k$ of $\nabla_{\frac{1}{2}} b_H(x)$.
\begin{lem}
\label{growthc}
Let $H \in C^{\infty}(\mathbb{H},\mathbb{R})$ be a Hamiltonian whose second order differential is bounded. For any $\frac{1}{2} \leq s \leq 1$ there exists a constant $C_s>0$ such that
\begin{align*}
\|\nabla_{\frac{1}{2}} b_H(x)^k \| \leq  \dfrac{C_s(1+|x|_{s})}{|k|^{s+1}}
\end{align*}
for any $k \in \mathbb{Z} \backslash \{ 0\}$ and any $x \in H^s (S^1,\mathbb{H})$.\\
Moreover if $y \in H^s (S^1,\mathbb{H})$ it holds
\begin{align*}
\|\nabla_\frac{1}{2} b_H(x)^k - \nabla_\frac{1}{2} b_H(y)^k \| \leq \dfrac{|\nabla H(x) - \nabla H (y) |_{s}}{ 2\pi |k|^{s+1}}.
\end{align*}
\end{lem}
\proof
For $k\neq 0$ we notice that
\begin{align}
\label{a2}
|P^k x|_{s}= \sqrt{2 \pi}|k|^s \|x^k\|. 
\end{align}
We compute
\begin{align*}
|P^k \nabla_s b_H(x) |_s&= \max_{|u|_s=1} \langle P^k \nabla_s b_H(x),u \rangle_s =\max_{|u|_s=1} \langle \nabla_s b_H(x),P^k u \rangle_s\\
&=\max_{|u|_s=1}  d b_H(x)[P^k u]  =\max_{|u|_s=1}  \int_0 ^1 \langle \nabla H(x)(t), (P^k u)(t) \rangle dt\\
&=\max_{\substack{y \in \mathbb{H} \\ 2 \pi |k|^{2s} \|y\|^2=1}}  \int_0 ^1 \langle \nabla H(x)(t), e^{2 \pi k t J} y \rangle dt\\
&= \max_{\substack{y \in \mathbb{H} \\ 2 \pi |k|^{2s} \|y\|^2=1}}  \int_0 ^1 \langle  e^{-2 \pi k t J}  \nabla H(x)(t),y \rangle dt\\
&=\max_{\substack{y \in \mathbb{H} \\ \|y\|=\dfrac{1}{\sqrt{2 \pi}|k|^s}}} \langle y , \int_0 ^1 e^{-2 \pi k t J}  \nabla H(x)(t) dt \rangle\\
&=\max_{\substack{y \in \mathbb{H} \\ \|y\|=\dfrac{1}{\sqrt{2 \pi}|k|^s}}} \langle y ,  \nabla H (x) ^k \rangle\\
&=\dfrac{1}{\sqrt{2 \pi}|k|^s} \| \nabla H(x) ^k\|\\
\end{align*}
where the second to last equality follows from the integral definition of the Fourier coefficients.
By \eqref{a2} it follows
\begin{align*}
|P^k \nabla_s b_H(x) |_s =\dfrac{\| \nabla H (x)^k \|}{\sqrt{2 \pi}|k|^s} =\dfrac{|P^k \nabla H(x)|_{s}}{2 \pi |k|^{2s}}.
\end{align*}
Using the integral definition of the $H^s$-norm and the assumptions on the boundedness of the second differential of $H$ we get the estimate
\begin{align*}
|P^k \nabla H(x)|_s &\leq |\nabla H (x)|_s \leq |\nabla H (x)+  \nabla H(0)- \nabla H(0)|_s \\
&\leq |\nabla H (x)-  \nabla H(0)|_s + |\nabla H(0)|_s  \leq c_s (1+ |x|_s),
\end{align*}
for any $s \leq 1$.\\ 
In conclusion, combining \eqref{a2} with the two estimates above we get
\begin{align}
\label{ccc}
\|\nabla_{s} b_H(x)^k \|= \dfrac{|P^k \nabla_{s} b_H(x)|_s }{2 \pi |k|^{s}}\leq \dfrac{C_s(1+|x|_{s})}{|k|^{3s}}.
\end{align}
For any $\frac{1}{2} \leq s \leq 1$ and any integer $k \neq 0$, it follows from Lemma \ref{lemgrad} that
\begin{align*}
\|\nabla_{s} b_H(x)^k \| =\dfrac{\|\nabla_\frac{1}{2} b_H(x)^k \|}{|k|^{2s-1}}
\end{align*}
hence using \eqref{ccc} we deduce
\begin{align*}
\|\nabla_\frac{1}{2} b_H(x)^k \| \leq \dfrac{C_s(1+ |x|_{s})}{|k|^{s+1}}.
\end{align*}
Performing an analogous computation we also get that if $x,y \in H^s(S^1,\mathbb{H})$ then
\begin{align*}
\|\nabla_\frac{1}{2} b_H(x)^k - \nabla_\frac{1}{2} b_H(y)^k \| \leq \dfrac{ |\nabla H(x) - \nabla H (y) |_{s}}{2 \pi|k|^{s+1}}.
\end{align*}
\endproof
In the case $\mathbb{H}= \mathbb{R}^{2n}$, under the assumption on the polynomial growth of the derivatives of $H$, it is possible to deduce that the map
\begin{align*}
E &\rightarrow E\\
x &\mapsto \nabla_{\frac{1}{2}} b_H(x)
\end{align*} 
is compact, namely that sends bounded sets into precompact sets (i.e. sets whose closure is compact).
This follows from the Sobolev's embedding theorem which implies that $H^{\frac{1}{2}}(S^1,\mathbb{R}^{n})$ is compactly embedded in $L^2(S^1,\mathbb{R}^n)$.
Nevertheless the following simple example, which serves as a prototype for the situation we will encounter in a more general case, shows that the map $x \mapsto \nabla_{\frac{1}{2}} b_H(x)$ is in general not compact if $\mathbb{H}$ is infinite dimensional.
\begin{ese}
\label{esempio}
Let us consider the quadratic time-independent Hamiltonian $H:\mathbb{H}\rightarrow \mathbb{R}$ defined as $H(y)=\mu \| y\| ^2$, with $\mu \in \mathbb{R}$.
We have that 
\begin{align*}
b_H(x)= \int_0 ^1 \mu \| x(t) \|^2 dt= \mu \langle x,x \rangle_{L^2}.
\end{align*}
By representing $x$ as a Fourier series we get
\begin{align*}
\langle x,x \rangle_{L^2}= \sum_{k\in \mathbb{Z}}  \langle x^k ,x^k \rangle,
\end{align*}
and 
\begin{align*}
 \langle x,x \rangle_{\frac{1}{2}} = \langle x^0,x^0\rangle + 2 \pi \sum_{k\neq 0} |k|  \langle x^k ,x^k \rangle.
\end{align*}
The $k$-th Fourier coefficient of $\nabla_{L^2} b_H(x)= \nabla H(x) $ is
\begin{align*}
\nabla H(x)^k= 2 \mu x^k,
\end{align*}
while the Fourier coefficients of $\nabla_{\frac{1}{2}} b_H(x)$ are
\begin{align*}
&\nabla_{\frac{1}{2}} b_H(x) ^0=2 \mu x^0,\\
&\nabla_{\frac{1}{2}} b_H(x) ^k = \dfrac{2 \mu}{2 \pi |k|} x^k \textrm{ for } k \neq 0,
\end{align*}
thus the map $x \mapsto \nabla_{\frac{1}{2}} b_H(x)$ is not compact as long as $\mu \neq 0$ because, for example, it sends any bounded non precompact set of constant curves into a non precompact sets of curves of $E$.
\end{ese}
We conclude the section by reviewing some general and well known results in critical point theory.\\
Let $\mathbb{H}$ be a separable Hilbert space and  $f\in C^1 (\mathbb{H},\mathbb{R)}$. 
\begin{defin}
A Palais-Smale sequence for $f$ at level $c\in \mathbb{R}$, abbreviated as (PS)$^c$, is a sequence $\{ x_n\}_{n \in \mathbb{N}}$ of elements of $\mathbb{H}$ such that $\nabla f(x_n) \underset{n \rightarrow + \infty} \longrightarrow 0$ in $\mathbb{H}$ and $f(x_n) \underset{n \rightarrow + \infty} \longrightarrow c$.
The map $f$ satisfies the \textit{Palais-Smale} condition at level $c$ if any (PS)$^c$ sequence of $f$ admits a convergent subsequence.
\end{defin} 
Let $\mathcal{F}$ be a family of subsets of $\mathbb{H}$, we define the minimax value $c(f,\mathcal{F})$ belonging to $f$ and $\mathcal{F}$ as
\begin{align*}
c(f,\mathcal{F}):= \inf_{F \in \mathcal{F}} \sup _{x \in F} f(x) \in \mathbb{R} \cup \{ - \infty \} \cup \{ + \infty \}.
\end{align*}
\begin{teo}[Minimax lemma]
\label{m} 
Assume $f$ and $\mathcal{F}$ meet the following conditions:
\begin{enumerate}[1)]
\item $\dot{x}= - \nabla f (x)$ defines a global flow $\varphi_t (x)$,
\item the family $\mathcal{F}$ is positively invariant under the flow, i.e., if $F \in \mathcal{F}$ then $\varphi_t (F) \in \mathcal{F}$ for every $t \geq 0$,
\item $-\infty< c(f,\mathcal{F}) < \infty$,
\end{enumerate}
then there exists a \emph{(PS)} sequence for $f$ at level $c(f,\mathcal{F})$. If in addition 
\begin{enumerate}[]
\item 4) $f$ satisfies \emph{(PS)} at level $c(f,\mathcal{F})$,
\end{enumerate}
then $f$ has a critical point with critical value  $c(f,\mathcal{F})$.
\end{teo}
\proof
See \cite{HZ94}, Section 3.2. 
\endproof

\section{Palais-Smale sequences of the Hamiltonian action functional}
\label{sec2}
In this section we study the properties of (PS)$^c$ sequences of the Hamiltonian action functional $\mathcal{A}_H:E \rightarrow \mathbb{R}$ under the assumption that the map
\begin{align*}
\mathbb{H} &\rightarrow \mathbb{H}\\
x &\mapsto \nabla H(x)
\end{align*}
is globally Lipschitz continuous.
The first result we obtain is that symplectomorphisms satisfying reasonable boundedness conditions induce a correspondence between bounded (PS)$^c$ sequences of the Hamiltonian action functional.
\begin{prop}
Let $\varphi:\mathbb{H}\rightarrow \mathbb{H}$ be a symplectic diffeomorphism such that $\varphi$ and $\varphi^{-1}$ have bounded differentials up to the second order, let $H \in C^{\infty}(\mathbb{H},\mathbb{R})$ be a Hamiltonian whose gradient is globally Lipschitz continuous and consider the Hamiltonian $G=H \circ \varphi^{-1}$.\\
Let $s > \frac{1}{2}$ be a real number, if $\{x_n\}_{n \in \mathbb{N}}\in E$ is a $H^s$-bounded \emph{(PS)}$^c$ sequence for $\mathcal{A}_H$ then $\{ \varphi_*(x_n)\}_{n \in \mathbb{N}}$ is a \emph{(PS)}$^c$ sequence for $\mathcal{A}_G$.
\label{PSinv}
\end{prop}
We remark that the assumption on the $H^s$-boundedness of (PS)$^c$ sequences with $s > \frac{1}{2}$ is not really restrictive, indeed in the final part of the section we will see that in our setting any $H^\frac{1}{2}$-bounded (PS)$^c$ sequence can be modified into a $H^s$-bounded (PS)$^c$ sequence for any $s < \frac{3}{2}$.\\ \\
Let $\varphi: \mathbb{H} \rightarrow \mathbb{H}$ be any symplectomorphism and consider the composition operator (known in the literature, in a more general context, as superposition operator) defined as 
\begin{align*}
\varphi_*: &E \rightarrow Im(\varphi_*)\\
&u \mapsto \varphi \circ u.
\end{align*}
The differential of $\varphi_*$ at a point $x \in E$, if it exists, is given by
\begin{align*}
\big( D \varphi_*(x)[u] \big)(t) &= \big( \dfrac{d}{ds}{\big{|}_{0}} \varphi_*(x + su)\big)(t) = \big( \dfrac{d}{ds}{\big{|}_{0}} \varphi(x(t) + su(t))\big) \\
&= D \varphi (x(t))[u(t)], \ \ \forall u \in E.
\end{align*}
We define the formal differential of $\varphi_*$ as the multiplication operator 
\begin{align*}
D \varphi_* (x)[u]=D \varphi (x)[u].
\end{align*}
\begin{lem}
Let $\varphi: \mathbb{H} \rightarrow \mathbb{H}$ be a diffeomorphism with bounded differentials up to the second order, then
\begin{align*}
\varphi_*: H^{\frac{1}{2}}(S^1,\mathbb{H}) \rightarrow H^{\frac{1}{2}}(S^1,\mathbb{H})
\end{align*}
is a continuous map.\\
Moreover for any $s>\frac{1}{2}$ there is a constant $C_s$ such that for every $x \in H^{s}(S^1,\mathbb{H})$ the operator $D \varphi_*(x)$ is bounded on $H^{\frac{1}{2}}(S^1,\mathbb{H})$, i.e.
\begin{align*}
|D \varphi_* (x)|_{\mathcal{L}(E,E)} \leq C_s (|x|_{H^{s}(S^1,\mathbb{H})} +1).
\end{align*}
\label{lemreg}
\end{lem}
\proof
The map $\varphi$ is globally Lipschitz thus $\varphi_* :L^2(S^1,\mathbb{H}) \rightarrow L^2(S^1,\mathbb{H})$ is globally Lipschitz; moreover it is possible to find a constant $K$ such that
\begin{align*}
| \varphi(x) |^2 _{sem, \frac{1}{2}} &= \int_0^1 \int_0^1 \dfrac{|\varphi(x(t))-\varphi(x(r))|^2}{|t-r|^2}dtdr \leq  \int_0^1 \int_0^1 K \dfrac{|x(t)-x(r)|^2}{|t-r|^2}dtdr \\
&=K \int_0^1\int_0^1 \dfrac{|x(t)-x(r)|^2}{|t-r|^2}dtdr\\
&= K | x |^2 _{sem,\frac{1}{2}}
\end{align*}
hence $\varphi_*: E \rightarrow E$ is well defined.\\
Our next step is to prove the continuity of $\varphi_* : H^1(\mathbb{D},\mathbb{H}) \rightarrow H^1(\mathbb{D},\mathbb{H})$, with $\mathbb{D}$ disk whose boundary is $S^1$.\\
The diffeomorphism $\varphi$ is $C^1$ and has bounded differential, thus the chain rule for Sobolev spaces $W^{1,p}(\mathbb{D},\mathbb{H})$ (which can be deduced using the standard chain rule for smooth maps and the dominated convergence theorem) implies that if $u \in H^1(\mathbb{D},\mathbb{H})$ then
\begin{align*}
 \left\lbrace \begin{array}{l}
\varphi_* (u)  \in H^1(\mathbb{D},\mathbb{H}) \\ 
\nabla (\varphi_*  (u))= D \varphi_* (u)  [ \nabla u].\\
   \end{array} \right.
\end{align*} 
Let us consider a sequence $\{ u_n\}_{n \in \mathbb{N}} \in H^1(\mathbb{D},\mathbb{H})$ such that $u_n  \underset{n \rightarrow + \infty}  \longrightarrow u$ in $H^1(\mathbb{D},\mathbb{H})$.
Since $u_n  \underset{n \rightarrow + \infty} \longrightarrow u$ in $H^1(\mathbb{D},\mathbb{H})$, up to the choice of a subsequence we have pointwise convergence almost everywhere, namely
\begin{align*}
 \left\lbrace \begin{array}{l}
 \ u_{n_k} (x) \  \underset{n_k \rightarrow + \infty} \longrightarrow u(x) \ \ \ \forall x \in \mathbb{D} \backslash \Sigma\\
\nabla u_{n_k} (x)  \underset{n_k \rightarrow + \infty} \longrightarrow \nabla u(x) \ \ \ \forall x \in \mathbb{D} \backslash \Sigma\\
   \end{array} \right.
\end{align*} 
 with $\Sigma$ null-subset of $\mathbb{D}$.\\
By the Lipschitz continuity of $\varphi$ and the dominated convergence theorem  it follows that
\begin{align*}
\varphi_* (u_n) \underset{n \rightarrow + \infty} \longrightarrow \varphi_* (u) \ \ \textrm{ in } L^2(\mathbb{D},\mathbb{H}).
\end{align*}
Using the continuity and boundedness assumption on the differential of $\varphi$ we deduce that 
\begin{align*}
D \varphi( u_{n_k} (x))[ \nabla u_{n_k}(x)] \underset{n_k \rightarrow + \infty} \longrightarrow D \varphi (u(x))[ \nabla u(x)] \ \ \ \forall x \in \mathbb{D} \backslash \Sigma.
\end{align*}
By the dominate convergence theorem we get 
\begin{align*}
D \varphi_*( u_{n_k})[ \nabla u_{n_k}] \underset{n_k \rightarrow + \infty} \longrightarrow D \varphi_* (u)[ \nabla u] \ \ \textrm{ in } L^2(\mathbb{D},\mathbb{H}),
\end{align*}
 and using chain rule it follows
\begin{align*}
\nabla(\varphi_* (u_{n_k})) \underset{n_k \rightarrow + \infty} \longrightarrow \nabla(\varphi_*  (u)) \ \ \textrm{ in } L^2(\mathbb{D},\mathbb{H}),
\end{align*}
hence the desired continuity of $\varphi_*$ in $H^1(\mathbb{D},\mathbb{H})$.\\
To recover the continuity of $\varphi_*$ on $H^{\frac{1}{2}}(S^1,\mathbb{H})$ we need to do a small digression.
It is known that given any element of $H^{\frac{1}{2}}(S^1,\mathbb{H})$ there is a unique extension of $u$ to an element $\tilde{u}$ of $H^1(\mathbb{D},\mathbb{H})$ such that $\partial \tilde{u}= u$ and $\Delta \tilde{u}=0$ in the interior of $\mathbb{D}$. The map
\begin{align*}
R: H^{\frac{1}{2}}(S^1,\mathbb{H}) &\rightarrow H^1(\mathbb{D},\mathbb{H}) \\
u &\mapsto \tilde{u} 
\end{align*}
is an isometry called harmonic extension and it is a right inverse of the trace operator
\begin{align*}
\partial: H^1(\mathbb{D},\mathbb{H}) &\rightarrow H^{\frac{1}{2}}(S^1,\mathbb{H})  \\
u &\mapsto u_{|\partial\mathbb{D}}.
\end{align*}
Both these linear maps are continuous.\\
Let us take a converging sequence $u_n  \underset{n \rightarrow + \infty} \longrightarrow u$ in $H^{\frac{1}{2}}(S^1,\mathbb{H})$.\\
It follows that
\begin{align*}
|\tilde{u}_n - \tilde{u}|_{H^1(\mathbb{D})} = |R(u_n - u)|_{H^1(\mathbb{D})} \leq C_1|u_n - u|_{H^{\frac{1}{2}}(S^1)} \underset{n \rightarrow + \infty} \longrightarrow 0,
\end{align*} 
and the trace theorem along with the continuity of $\varphi_* : H^1(\mathbb{D},\mathbb{H}) \rightarrow H^1(\mathbb{D},\mathbb{H})$ imply
\begin{align*}
|\varphi_* (u_n) - \varphi_*(u)|_{H^{\frac{1}{2}}(S^1)}  = |\partial \big(\varphi_* (\tilde{u}_n) - \varphi_* (\tilde{u})\big)|_{H^{\frac{1}{2}}(S^1)} \leq C_2 |\varphi_* (\tilde{u}_n) - \varphi_* (\tilde{u})|_{H^1(\mathbb{D})} \underset{n \rightarrow + \infty} \longrightarrow 0,
\end{align*} 
hence the continuity of $\varphi_* :E \rightarrow E$.\\
Our next goal is to prove the existence of the differential of $\varphi_*: H^{\frac{1}{2}}(S^1,\mathbb{H}) \rightarrow H^{\frac{1}{2}}(S^1,\mathbb{H})$ at points belonging to $H^{s}(S^1,\mathbb{H})$ with $s > \frac{1}{2}$.\\
Using the fact that the harmonic extension of a loop $u$ minimizes the Dirichlet integral among all functions with trace $u$, together with the boundedness assumptions on $D \varphi$ and on $D^2 \varphi$ we compute
\begin{align*}
|D \varphi_* (u)[v]|_{H^{\frac{1}{2}}(S^1)} & \leq C_1 |R D \varphi_* (u)[v]|_{H^1(\mathbb{D})}\\
&\leq C_1 |D \varphi_* (\tilde{u})[\tilde{v}]|_{H^1(\mathbb{D})}\\
&= C_1(|D \varphi_* (\tilde{u})[\tilde{v}]|_{L^2(\mathbb{D})}  + |D(D \varphi_* (\tilde{u})[\tilde{v}])|_{L^2(\mathbb{D})})\\
&\leq C_1(C_2 |\tilde{v}|_{L^2(\mathbb{D})} +  |D^2 \varphi (\tilde{u})[\nabla \tilde{u}, \tilde{v}]|_{L^2(\mathbb{D})} + |D \varphi (\tilde{u})[\nabla \tilde{v}]|_{L^2(\mathbb{D})}) \\
&\leq C_1(C_2 |\tilde{v}|_{L^2(\mathbb{D})} + C_3 |\nabla \tilde{u}|_{L^p(\mathbb{D})}  |\tilde{v}|_{L^q(\mathbb{D})} +  C_3|\nabla \tilde{v}|_{L^2(\mathbb{D})})
\end{align*}
where by the H\"older's inequality the last estimate holds for $p>2$ arbitrary and $q$ equal to twice the conjugate exponent of $\frac{p}{2}$.\\
Since $|\tilde{v}|_{L^2(\mathbb{D})}$, $|\tilde{v}|_{L^q(\mathbb{D})}$ and $|\nabla \tilde{v}|_{L^2(\mathbb{D})}$ can be bounded by the quantity $|\tilde{v}|_{H^1(\mathbb{D})}$, which is bounded by the value $|v|_{H^\frac{1}{2}(S^1)}$ multiplied with a constant, we infer that
\begin{align}
\label{a3}
|D \varphi_* (u)[v]|_{H^{\frac{1}{2}}(S^1)} \leq C (|\nabla \tilde{u}|_{L^p(\mathbb{D})} +1 )|v|_{H^\frac{1}{2}(S^1)}.
\end{align} 
If we choose $t \in (1,2)$ and we take $p= 1 + \frac{t}{2-t} >2$, then the Sobolev's embedding theorem implies that $H^t(\mathbb{D},\mathbb{H})$ embeds continuously into $W^{1,p}(\mathbb{D},\mathbb{H})$, therefore
\begin{align*}
|D \varphi_* (u)[v]|_{H^{\frac{1}{2}}(S^1)} &\leq C_t (|\nabla \tilde{u}|_{H^t(\mathbb{D})} +1 )|v|_{H^\frac{1}{2}(S^1)}\\
& \leq  C_t (|\nabla u|_{H^{t-\frac{1}{2}}(S^1)} +1 )|v|_{H^\frac{1}{2}(S^1)},
\end{align*}
where the last inequality follows since the harmonic extension is a bounded right inverse of the trace operator $\partial : H^t(\mathbb{D},\mathbb{H}) \rightarrow H^{t-\frac{1}{2}}(S^1,\mathbb{H})$.
\endproof
We are now ready to prove Proposition \ref{PSinv}.
\proof
The elements of $H^1(S^1,\mathbb{H})$ are absolutely continuous loops, thus their action is preserved by symplectomorphisms between simply connected domains. Indeed if $x$ is an absolutely continuous loop then
\begin{align*}
-\frac{1}{2}\int_0 ^1 \langle J (\varphi \circ x)'(t), \varphi(x(t)) \rangle dt &= \int_{S^1} x^* (\varphi^* \lambda) = \int_{S^1} x^* (\lambda + dh)= \int_{S^1} x^* (\lambda)\\
&=-\frac{1}{2}\int_0 ^1 \langle J x'(t), x(t)\rangle dt,
\end{align*}
where $\varphi^* \lambda- \lambda$ is exact and therefore it is the differential of a smooth function $h$.
The Hamiltonian part of the action functional does not change, since by definition
\begin{align*}
\int_0 ^1 G(\varphi (x(t)))dt =\int_0 ^1 H(x(t))dt,
\end{align*}
thus we conclude that 
\begin{align}
\label{f}
\mathcal{A}_G(\varphi_*(x)) = \mathcal{A}_H(x).
\end{align}
The density of $H^1(S^1,\mathbb{H}) \hookrightarrow H^\frac{1}{2}(S^1,\mathbb{H})$ together with the continuity of the map $\varphi_*:E \rightarrow E$ and of $\mathcal{A}_H$ imply that the equality \eqref{f} holds for any loop $x \in E$. In particular, given any (PS)$^c$ sequence $\{x_n\}_{n \in \mathbb{N}}\in E$ we have
\begin{align*}
\mathcal{A}_G(\varphi_*(x_n)) = \mathcal{A}_H(x_n),
\end{align*}
for any $n$. Let $\{x_n\}_{n \in \mathbb{N}}$ be a $H^s$-bounded (PS)$^c$ sequence with $s> \frac{1}{2}$; if we differentiate the equation above, in view of Lemma \ref{lemreg} we can write
\begin{align*}
d \mathcal{A}_G (\varphi_* (x_n)) = d \mathcal{A}_H (x_n) D \varphi_* ^{-1} (\varphi_*(x_n)),
\end{align*}
thus
\begin{align*}
\nabla_{\frac{1}{2}} \mathcal{A}_G (\varphi_* (x_n)) =  D \varphi_* ^{-T} (\varphi_*(x_n)) \nabla_{\frac{1}{2}}  \mathcal{A}_H (x_n).
\end{align*}
Lemma \ref{lemreg} also implies the boundedness of 
\begin{align*}
D \varphi_* ^{-1}: x \mapsto \mathcal{L}(E,E),
\end{align*}
thus we can find a constant $c$ such that
\begin{align*}
|\nabla_{\frac{1}{2}} \mathcal{A}_G (\varphi_* (x_n))|_{\frac{1}{2}} \leq | D \varphi_* ^{-T} (\varphi_*(x_n))|_{\mathcal{L}(E,E)} | \nabla_{\frac{1}{2}}  \mathcal{A}_H (x_n)|_{\frac{1}{2}}  \leq c | \nabla_{\frac{1}{2}}  \mathcal{A}_H (x_n)|_{\frac{1}{2}}.
\end{align*}
This, combined with equality \eqref{f} imply that if $\{ x_n\}_{n \in \mathbb{N}}$ is an $H^s$-bounded (PS)$^c$ sequence of $\mathcal{A}_H$, then $\{ \varphi_*(x_n)\}_{n \in \mathbb{N}}$ is a (PS)$^c$ sequence of $\mathcal{A}_G$.
\endproof
The next two results serve to approximate bounded (PS) sequences with equivalent (PS) sequences fulfilling stricter boundedness conditions (we say that two (PS)$^c$ sequences $\{x_n\}_{n \in \mathbb{N}}$ and $\{y_n\}_{n \in \mathbb{N}}$ are equivalent if $\|x_n - y_n\|  \underset{n \rightarrow + \infty} \longrightarrow 0 $ and they share the same almost critical level $c$).
\begin{lem}
\label{equiva}
Let $H \in C^{\infty}(\mathbb{H},\mathbb{R})$ be a Hamiltonian whose second order differential is bounded.
Let $\frac{1}{2} \leq r \leq 1$ and $\{x_n\}_{n \in \mathbb{N}}$ be a $H^r$-bounded \emph{(PS)}$^c$ sequence of $\mathcal{A}_H:E \rightarrow \mathbb{R}$ such that $\nabla_{\frac{1}{2}} \mathcal{A}_H(x_n)$ is $H^r$-infinitesimal. Then the elements
\begin{align*}
y_n:= x_n - (P^+ - P^-)\nabla_{\frac{1}{2}} \mathcal{A}_H(x_n) 
\end{align*}
define a $H^s$-bounded \emph{(PS)}$^c$ sequence $\{y_n\}_{n \in \mathbb{N}}$ for any $s<r + \frac{1}{2}$, such that
\begin{align*}
|y_n-x_n|_{r}  = o(1).
\end{align*}
Moreover the following estimate holds
\begin{align*}
|\nabla_{\frac{1}{2}} \mathcal{A}_H(y_n)|_s \leq o(1) + c_s |\nabla H(x_n) - \nabla H(y_n) |_r.
\end{align*}
\end{lem}
\proof
Let $R>0$ be an upper bound for $|x_n|_r$. By assumption the sequence
\begin{align}
\label{a1}
\nabla_{\frac{1}{2}} \mathcal{A}_H(x_n) = (P^+ - P^-)x_n - \nabla_{\frac{1}{2}} b_H(x_n)
\end{align}
is infinitesimal in $H^r(S^1,\mathbb{H})$.\\
The sequence given by
\begin{align*}
y_n:= x_n - (P^+ - P^-)\nabla_{\frac{1}{2}} \mathcal{A}_H(x_n) 
\end{align*}
is (PS)$^c$ because $\mathcal{A}_H$ and its gradient are uniformly continuous on bounded sets.\\
For any $k \in \mathbb{Z}$, if we apply the projector $P^k$ to \eqref{a1} we get 
\begin{align*}
P^k \nabla_{\frac{1}{2}} \mathcal{A}_H(x_n) =(\textrm{sgn} k) P^k x_n  -P^k \nabla_{\frac{1}{2}} b_H(x_n),
\end{align*}
thus
\begin{align*}
P^k y_n =P^k x_n - (\textrm{sgn} k)P^k \nabla_{\frac{1}{2}} \mathcal{A}_H(x_n) = (\textrm{sgn} k)P^k \nabla_{\frac{1}{2}} b_H(x_n),
\end{align*}
hence
\begin{align*}
y_n^k = (\textrm{sgn} k) \nabla_{\frac{1}{2}} b_H(x_n) ^k.
\end{align*}
By Lemma \ref{growthc}, for any $k \neq 0$ we have the estimate
\begin{align*}
\|y_n^k\| \leq \frac{c (R+1)}{|k|^{r+1}},
\end{align*}
therefore 
\begin{align*}
|y_n|_{H^s} ^2 = \|y_n ^0\|^2 + 2 \pi \sum_{k \neq 0} |k|^{2s} \|y_n ^k\|^2 \leq R^2 + 2 \pi c^2(R+1)^2 \sum_{k \neq 0} \frac{1}{|k|^{2(r-s+1)}}
\end{align*}
where the series above converges for any $s<r +\frac{1}{2}$, thus we get the desired bound.\\
To prove the second part of the statement we rewrite the $H^\frac{1}{2}$-gradient of $\mathcal{A}_H$ as
\begin{align*}
\nabla_{\frac{1}{2}} \mathcal{A}_H (y_n) &= (P^+ -P^-)(x_n - (P^+ - P^-) \nabla _{\frac{1}{2}} \mathcal{A}_H (x_n)) - \nabla_{\frac{1}{2}} b_H(y_n)\\
&=(P^+ -P^-)x_n - (P^+ + P^-) \nabla _{\frac{1}{2}} \mathcal{A}_H (x_n) - \nabla_{\frac{1}{2}} b_H(y_n)\\
&=(P^+ -P^-)x_n - \nabla _{\frac{1}{2}} \mathcal{A}_H (x_n) +P^0 \nabla _{\frac{1}{2}} \mathcal{A}_H (x_n) - \nabla_{\frac{1}{2}} b_H(y_n)\\
&=(P^+ -P^-)x_n - (P^+ -P^-)x_n  +\nabla_{\frac{1}{2}} b_H(x_n) +P^0 \nabla _{\frac{1}{2}} \mathcal{A}_H (x_n) - \nabla_{\frac{1}{2}} b_H(y_n)\\
&= P^0 \nabla _{\frac{1}{2}} \mathcal{A}_H (x_n) +\nabla_{\frac{1}{2}} b_H(x_n) - \nabla_{\frac{1}{2}} b_H(y_n).
\end{align*}
Using the second estimate of Lemma \ref{growthc} we deduce 
\begin{align*}
|\nabla_{\frac{1}{2}} \mathcal{A}_H (y_n)|_s ^2 &\leq |P^0 \nabla_{\frac{1}{2}} \mathcal{A}_H (x_n)|_s ^2 + 2 \pi \sum_{k \neq 0} |k|^{2s} \|\nabla_\frac{1}{2} b_H(x_n)^k - \nabla_\frac{1}{2} b_H(y_n)^k \|^2 \\ 
&+  \|\nabla_\frac{1}{2} b_H(x_n)^0 - \nabla_\frac{1}{2} b_H(y_n)^0 \|^2\\
&\leq  \|P^0 \nabla_{\frac{1}{2}} \mathcal{A}_H (x_n)\|^2 + c |\nabla H(x_n) - \nabla H (y_n) |_{r} ^2  \sum_{k \neq 0} \dfrac{1}{|k|^{2(r-s+1)}}.
\end{align*}
The first quantity in the last line is infinitesimal since $\{x_n \}_{n \in \mathbb{N}}$ is (PS), and the series appearing in the second quantity converges for any $s<r+ \frac{1}{2}$, thus we get the desired estimate.
\endproof
\begin{lem}
Let $H \in C^{\infty}(\mathbb{H},\mathbb{R})$ be a Hamiltonian whose second and third order differentials are bounded.
Let $\{x_n\}_{n \in \mathbb{N}}$ be a bounded  \emph{(PS)}$^c$ sequence of $\mathcal{A}_H$, then for any $\frac{1}{2} \leq s < \frac{3}{2}$ we can find an equivalent \emph{(PS)}$^c$ sequence $\{y_n\}_{n \in \mathbb{N}}$ which is $H^s$-bounded and such that $\nabla_{\frac{1}{2}} \mathcal{A}_H (y_n)$ is infinitesimal in the $H^s$-norm.
\end{lem}
\proof
We start by applying  the first part of Lemma \ref{equiva} to obtain, for $s<1$, a $H^s$-bounded (PS)$^c$ sequence $\{z_n\}_{n \in \mathbb{N}}$ which is equivalent to $\{x_n\}_{n \in \mathbb{N}}$ and then we modify this new sequence to get an equivalent one
\begin{align*}
u_n:= z_n - (P^+ - P^-)\nabla_{\frac{1}{2}} \mathcal{A}_H(z_n).
\end{align*}
By Lemma \ref{equiva} the sequence $\{u_n\}_{n \in \mathbb{N}}$ is $H^s$-bounded and $\nabla_{\frac{1}{2}} \mathcal{A}_H (u_n)$ is $H^s$-infinitesimal for any $s< 1$, where the latter claim follows by using the mean value theorem and the fact that the composition operator 
\begin{align*}
H^{\frac{1}{2}}(S^1,\mathbb{H}) &\rightarrow H^{\frac{1}{2}}(S^1,\mathbb{H})\\
u &\mapsto \nabla H(u)
\end{align*}
has bounded differential if restricted to elements of $H^{r}(S^1,\mathbb{H})$ with $r>\frac{1}{2}$ (for this we have to require the boundedness of the second and third order differentials of $H$; see the analogous proof we gave for the operator $\varphi_*$).\\
Now we introduce another equivalent sequence
\begin{align*}
v_n:= u_n - (P^+ - P^-)\nabla_{\frac{1}{2}} \mathcal{A}_H(u_n),
\end{align*}
which is bounded in the $H^{s+ \frac{1}{2}}$-norm for any $s <1$ because of the first part of Lemma \ref{equiva}, and for which $\nabla_{\frac{1}{2}} \mathcal{A}_H(v_n)$ is $H^s$-infinitesimal.\\ 
Finally we define another equivalent sequence
\begin{align*}
y_n:= v_n - (P^+ - P^-)\nabla_{\frac{1}{2}} \mathcal{A}_H(v_n).
\end{align*}
The composition operator 
\begin{align*}
H^{1}(S^1,\mathbb{H}) &\rightarrow H^{1}(S^1,\mathbb{H})\\
u &\mapsto \nabla H(u)
\end{align*}
has bounded differential if restricted to elements of $H^{t}(S^1,\mathbb{H})$ with $t>1$, therefore by interpolation with the composition operator defined on $H^\frac{1}{2}(S^1,\mathbb{H})$ we get that for any $\frac{1}{2} \leq q \leq 1$, the operator 
\begin{align*}
H^{q}(S^1,\mathbb{H}) &\rightarrow H^{q}(S^1,\mathbb{H})\\
u &\mapsto \nabla H(u)
\end{align*}
has bounded differential if restricted to elements of $H^{t}(S^1,\mathbb{H})$, indeed by the interpolation formula we have
\begin{align*}
|D \nabla H (u)|_{\mathcal{L}(H^q,H^q)} \leq |D\nabla H (u)|_{\mathcal{L}(H^\frac{1}{2},H^\frac{1}{2})} ^\theta |D \nabla H (u)|_{\mathcal{L}(H^1,H^1)} ^{1- \theta},
\end{align*}
for any $q=1- \dfrac{\theta}{2}$ with $\theta \in (0,1)$. As a corollary of Lemma \ref{equiva} we deduce that the sequence $\nabla_{\frac{1}{2}} \mathcal{A}_H (y_n)$ is $H^s$-infinitesimal for any $s< \frac{3}{2}$.
\endproof
As an aside remark we observe that (PS)$^c$ sequence $\{y_n\}_{n \in \mathbb{N}}$ appearing in the statement of the proposition above can be made $H^s$-bounded for any $s>0$, if additional conditions on the boundedness of the higher order differentials of $H$ are fulfilled.\\
The $L^2$-gradient of $\mathcal{A}_H$ is readily computed as
\begin{align*}
\nabla_{L^2} \mathcal{A}_H (x) = -\big(J\dot{x} + \nabla H(x)\big),
\end{align*}
in fact
\begin{align*}
d\mathcal{A}_H (x) [u]&= \dfrac{d}{ds} \big{|}_{s=0} \mathcal{A}_H (x + su) \\
&= -\frac{1}{2} \int_0 ^1 \langle J \dot{u} (t), x(t) \rangle dt -\frac{1}{2} \int_0 ^1 \langle J \dot{x} (t), u(t)\rangle dt - \int_0 ^1 dH(x(t))[u(t)] dt \\
&= -\int_0 ^1 \langle J \dot{x}(t) + \nabla H(x(t)), u(t) \rangle dt.
\end{align*}
If $\{ x_n\}_{n \in \mathbb{N}}$ is a (PS)$^c$ sequence found by means of the proposition above we get that $\nabla_{L^2} \mathcal{A}_H (x_n)$ is infinitesimal in the $L^2$-norm, indeed using Lemma \ref{lemgrad} we deduce
\begin{align*}
|\nabla_{\frac{1}{2}} \mathcal{A}_H (x_n)|_1 \leq | \nabla_{L^2} \mathcal{A}_H (x_n)|_{L^2} \leq 2 \pi |\nabla_{\frac{1}{2}} \mathcal{A}_H (x_n)|_1,
\end{align*}
thus the loop
\begin{align*}
z_n(t):=J \dot{x}_n(t) + \nabla H(x_n(t)) 
\end{align*}
is infinitesimal in the $L^2$-norm.
\begin{prop}\label{lemPS}
Let $H \in C^{\infty}(\mathbb{H},\mathbb{R})$ be a Hamiltonian whose second and third order differentials are bounded.
Let $\{x_n\}_{n \in \mathbb{N}}$ be a bounded \emph{(PS)}$^c$ sequence of $\mathcal{A}_H$, then we can find an equivalent \emph{(PS)}$^c$ sequence $\{y_n\}_{n \in \mathbb{N}}$ which is $H^s$-bounded for any $s < \frac{3}{2}$ and such that 
\begin{align*}
|J \dot{y}_n + \nabla H(y_n)|_{L^2} \underset{n \rightarrow + \infty} \longrightarrow 0.
\end{align*} 
Moreover we have that
\begin{align*}
H(y_n) \underset{n \rightarrow + \infty} \longrightarrow c \in \mathbb{R} \cup \{\pm \infty\}
\end{align*} 
uniformly to some constant function $c$.
\end{prop}
\proof
The first part of the result follows from the discussion above.\\
If we differentiate $H \circ y_n$ with respect to $t$ we get 
\begin{align*}
\dfrac{d}{dt} H(y_n(t))&= d H(y_n(t))[ \dot{y}_n(t)]= \langle \nabla H (y_n(t)), \dot{y}_n(t)  \rangle\\
&=\langle z_n(t) - J\dot{y}_n(t) , \dot{y}_n(t)  \rangle\\
&=\langle z_n(t) , \dot{y}_n(t)  \rangle,
\end{align*}
with $z_n$ loop which is infinitesimal in the $L^2$-norm.
Since $\{\dot{y}_n \}_{n \in \mathbb{N}}$ is $L^2$-bounded, we deduce that $\dfrac{d}{dt} H(y_n)$ defines an infinitesimal sequence in the $L^1$-norm and because $y_n$ is absolutely continuous the conclusion follows.
\endproof

\section{Non-squeezing theorem}
\label{sec3}
Let $(\mathbb{H},\omega)$ be a symplectic Hilbert space endowed with a compatible inner product $\langle \cdot, \cdot \rangle$ and let $\{e_i,f_i\}_{i \in \mathbb{N}}$ be a countable orthonormal basis such that any $\{e_i,f_i\}$ spans a symplectic plane. For any $n \in \mathbb{N}$ we consider the orthogonal projections
\begin{align}
\label{proiez}
\begin{split}
P_n :\mathbb{H} &\rightarrow \mathbb{H}_n\\
 x& \mapsto \sum_{i=1}^{n} \langle x,e_i \rangle e_i +  \langle x,f_i \rangle f_i,
 \end{split}
\end{align}  
onto the $2n$-dimensional symplectic Hilbert subspace $\mathbb{H}_n$.\\
We define a set of admissible symplectomorphisms
\begin{align*}
Symp_a (\mathbb{H},\omega,\langle \cdot, \cdot \rangle):=& \big{\{} \varphi \in Symp(\mathbb{H},\omega) \ \big{|}  \textrm{ for } k=\pm 1, \ D \varphi^{k} \textrm{ and } D^2 \varphi^{k} \textrm{ are bounded}; \\
& (I-P_n) {\varphi^k}_{\vert P_n \mathbb{H}} \underset{n \rightarrow + \infty} \longrightarrow 0 \textrm{ uniformly on bounded sets}; \\
& [P_n , D \varphi^k (x)^*]  \underset{n \rightarrow + \infty} \longrightarrow 0 \textrm{ in operators' norm, uniformly in} \ x \in \mathbb{H}\\ 
&\textrm{on bounded sets} \big{\}}.
\end{align*}
which is easy to check that is actually a group.
\begin{thm*}[Infinite dimensional non-squeezing]
Let $(\mathbb{H},\omega)$ be a Hilbert symplectic space endowed with a compatible inner product $\langle \cdot, \cdot \rangle$, $B_r$ the ball centred in $0$ with radius $r$ and $Z_R$ a cylinder whose basis lays on a symplectic plane and has symplectic area $\pi R^2$. Let $\varphi \in Symp_a (\mathbb{H},\omega,\langle \cdot, \cdot \rangle)$, if $\varphi(B_r) \subset Z_R$ then $r\leq R$.
\end{thm*}
\begin{proof}[Proof's outline]
The cylinder $Z_R$ can be written as
\begin{align*}
Z_R= \{ z \in \mathbb{H} \ | \ x_1^2 + y_1^2 < R^2  \}.
\end{align*}
The non-squeezing theorem we want to prove is equivalent to the claim that the ball $B_s$ with $s>1$ cannot be symplectically embedded by $\varphi \in Symp_a (\mathbb{H},\omega,\langle \cdot, \cdot \rangle)$ into the cylinder $Z_1$. Indeed if any admissible symplectomorphism squeezes $B_{\frac{r}{R}}$ into $Z_1$, then the symplectomorphism $\phi \in Symp_a (\mathbb{H},\omega,\langle \cdot, \cdot \rangle)$ defined as $\phi = R \phi R^{-1}$ squeezes $B_r$ into $Z_R$ and vice-versa if any $\phi$ squeezes $B_r$ into $Z_R$ then $\varphi = R^{-1} \varphi R$ squeezes $B_{\frac{r}{R}}$ into $Z_1$.\\
Our strategy to prove the above reformulation of the non-squeezing theorem is assuming the existence for $r>1$ of an admissible symplectic embedding $\varphi: B_r \hookrightarrow Z_1$ and then showing that this leads to a contradiction.\\
Since $r>1$ we can choose two real numbers $m>\pi$, $\delta>0$ and a smooth map $g:[0,+\infty[ \rightarrow \mathbb{R}$ such that
\begin{align*}
 \left\lbrace \begin{array}{l}
 g(t)=0  \ \ \ \ \  \textrm{  \ \ \ if } t <  \delta, \\
 g(t)= m  \ \ \ \  \textrm{  	\ \ \ if } t \geq r- \delta,\\
 0 \leq g '(t) < \pi.
 \end{array} \right.
\end{align*}
We define a time-independent Hamiltonian $F:\mathbb{H} \rightarrow \mathbb{R}$ as 
\begin{align*}
F(x):= g (|x|^2).
\end{align*}
Its Hamiltonian flow is supported in $B_r$ and Proposition \ref{propF} will show that the only bounded (PS)$^c$ sequences of $\mathcal{A}_F$ are at non positive levels $c$.
Nevertheless Proposition \ref{propns} will exhibit a (PS)$^c$ sequence of $\mathcal{A}_F$ at positive level, hence the initial assumption $\varphi(B_r) \subset Z_1$ leads to a contradiction.
\end{proof}
To complete the proof we need to show that Proposition \ref{propF} and Proposition \ref{propns} hold.
\begin{prop}
\label{propF} The action functional $\mathcal{A}_F:E\rightarrow \mathbb{R}$ associated to the time-independent Hamiltonian $F(x)=g(|x| ^2)$ has no bounded \emph{(PS)}$^c$ sequences at any positive level $c$.
\end{prop}
\proof
We know that
\begin{align*}
\nabla F(x)= 2 g '(|x|^2) x,
\end{align*}
and by the properties of $g$ there is a real number $\epsilon>0$ such that 
\begin{align*}
g'(t) < \pi - \epsilon, \textrm{ for any } t.
\end{align*}
Given any bounded (PS)$^c$ sequence of loops, according to Proposition \ref{lemPS} we can find  an equivalent $H^1$-bounded (PS)$^c$ sequence $\{x_n\}_{n \in \mathbb{N}}$ such that
\begin{align*}
g(|x_n(t)|^2)=F(x_n(t)) \underset{n \rightarrow + \infty} \longrightarrow c_1 \in \mathbb{R}
\end{align*}
uniformly in $t$, hence we deduce the uniform convergence of
\begin{align*}
g ' (|x_n(t)|^2) \underset{n \rightarrow \infty}\longrightarrow  d  \in \mathbb{R}.
\end{align*}
This implies that
\begin{align*}
\nabla F(x_n(t)) -2  d x _n(t)  \underset{n \rightarrow \infty}\longrightarrow  0 \ \ \textrm{ uniformly in t},
\end{align*}
and hence
\begin{align*}
\nabla F(x_n) -2  d x _n  \underset{n \rightarrow \infty}\longrightarrow  0 \ \ \textrm{ in } L^2(S^1,\mathbb{H}).
\end{align*}
From Lemma \ref{lemgrad} it follows that
\begin{align*}
|\nabla_{\frac{1}{2}} b_F(x_n) -2  d T^*x _n|_{\frac{1}{2}} = |T^*(\nabla F(x_n) -2  d x _n)|_{\frac{1}{2}} \leq  |\nabla F(x_n) -2  d x _n|_{L^2},
\end{align*}
thus
\begin{align*}
\nabla_{\frac{1}{2}} b_F(x_n) -2  d T^*x _n \underset{n \rightarrow \infty}\longrightarrow  0 \ \  \textrm{ in } H^{\frac{1}{2}}(S^1,\mathbb{H}).
\end{align*}
In order for $x_n=(x_n ^- ,x_n ^0 ,x_n ^+)$ to define a (PS)$^{c}$ sequence for $\mathcal{A}_{F}$ it is necessary that $\nabla_{\frac{1}{2}} a(x_n) - \nabla_{\frac{1}{2}} b_{F}(x_n)  \underset{n \rightarrow \infty}{\longrightarrow} 0$ in the $H^{\frac{1}{2}}$-norm, thus it is necessary that
\begin{align}
\label{mmm}
x_n^+ - x_n ^- - 2  d T^*x _n = \nabla_{\frac{1}{2}} a(x_n) - 2  d T^*x _n   \underset{n \rightarrow \infty}{\longrightarrow} 0 \ \  \textrm{ in } H^{\frac{1}{2}}(S^1,\mathbb{H}).
\end{align}
Because of Lemma \ref{lemgrad} we know that 
\begin{align*}
& 2 d T^*x _n^0 =   2 d x^0 _n,\\
&2 d T^*x _n^k = \dfrac{2 d}{2 \pi |k|} x_n ^k  \textrm{ for } k \neq 0,
\end{align*}
thus the inequality $d < \pi - \epsilon$ implies that
\begin{align*}
\frac{\epsilon}{\pi} |x_n^+|_{\frac{1}{2}}  & = |x_n^+  - \frac{2  (\pi- \epsilon)}{2 \pi} x _n ^+|_{\frac{1}{2}} \leq |x_n^+  - 2  d T^*x _n ^+|_{\frac{1}{2}} ,  \\
|x_n^-|_{\frac{1}{2}}   &\leq |x_n^-  + 2  d T^*x _n ^-|_{\frac{1}{2}} .
\end{align*}
Therefore \eqref{mmm} is possible only if $|x_n ^+|_\frac{1}{2} ^2  \underset{n \rightarrow \infty}\longrightarrow 0$ and $|x_n ^-|_\frac{1}{2} ^2   \underset{n \rightarrow \infty}\longrightarrow 0$; since 
\begin{align*}
\mathcal{A}_F(x_n)= \dfrac{1}{2} |x^+_n|_{\frac{1}{2}}^2 - \dfrac{1}{2} | x^-_n|_{\frac{1}{2}}^2- b_F(x_n),
\end{align*}
and 
\begin{align*}
0 \leq b_F(x) \leq m, \ \ \  \forall x \in E,
\end{align*}
we deduce that
\begin{align*}
\mathcal{A}_F(x_n) \underset{n \rightarrow \infty}\longrightarrow  c,
\end{align*}
with $-m \leq c \leq 0$.
\endproof
The last step to prove the non-squeezing is to show that if squeezing were possible, then we would be able to find a (PS)$^{c}$ sequence at level $c>0$ for the action functional $\mathcal{A}_F$.
In order to do this we first introduce the concept of approximation scheme as in \cite{CLL97} and \cite{Abb01}.
\begin{defin}
Let $X$ be a separable Hilbert space. An \textit{approximation scheme} with respect to a bounded linear operator $L \in \mathcal{L}(X,X)$ is a sequence $\{P_n\}_{n \in \mathbb{N}}$ of finite dimensional orthogonal projections such that
\begin{enumerate}[1)]
\item $rank(P_n) \subset rank(P_m)$ if $n \leq m$,
\item $P_n \underset{n \rightarrow + \infty} \longrightarrow I$ strongly,
\item $[P_n,L]:=P_n L -L P_n \underset{n \rightarrow + \infty} \longrightarrow 0$ in the operators' norm.
\end{enumerate}
\end{defin}
An explicit example of approximation scheme for the identity map is given by the sequence $\{P_n\}_{n \in \mathbb{N}}$ of orthogonal projections as in \eqref{proiez}.\\
Let $\varphi \in Symp_a (\mathbb{H},\omega,\langle \cdot, \cdot \rangle)$ be such that $\varphi(B_r) \subset Z_1$ with $r>1$, we define a Hamiltonian $H: \mathbb{H} \rightarrow \mathbb{R}$ as
\begin{align*}
 H(x):=\left\lbrace \begin{array}{l}
 F(\varphi^{-1}(x))  \   \textrm{  \ \ if } x \in \varphi (B_r), \\
m  \ \ \ \ \  \textrm{  \ \ \ \ \ \ \ \ \ if } x \notin \varphi (B_r).
   \end{array} \right.
\end{align*}
The support of the Hamiltonian $H-m$ is the set 
\begin{align*}
\overline{\{ H < m \}} \subsetneq \varphi (B_r)
\end{align*}
which has positive distance from $\partial \overline{ \varphi (B_r)}$.\\
Indeed $D\varphi^{-1}$ is bounded by assumption, hence by the mean value theorem for any points $x,y \in B_r$ we have  
\begin{align*}
\| x - y \| = \| \varphi^{-1}(\varphi(x)) - \varphi^{-1}(\varphi(y)) \| \leq C \| \varphi(x) - \varphi(y) \|
\end{align*}
with $C$ positive constant, thus 
\begin{align*}
\frac{\|x - y \|}{C} \leq \| \varphi(x) - \varphi(y) \|,
\end{align*}
and this implies that if $\| x-y\| \geq \delta_0$, then $\| \varphi(x)-\varphi(y)\| \geq \frac{\delta_0}{C}$.\\
Let $x\in B_r$, if the distance $d(x,\partial \overline{B_r})$ is not larger than the value $\delta >0$ appearing in the definition of the function $F$, then $F(x)=m$.\\
Therefore for any point $y \in \varphi(B_r)$ such that $d(y,\partial \overline{\varphi(B_r)}) \leq \frac{\delta}{C}$, we get $H(y)=m$ and since $\varphi(\partial \overline{B_r}) = \partial (\varphi(\overline{B_r}))$ we deduce the claim, namely the existence of a real value $\lambda>0$ such that 
\begin{align*}
\overline{\{ H < m \}} +B_\lambda  &\subsetneq  \varphi(B_r).
\end{align*}
We define a quadratic form $q: \mathbb{H}\rightarrow \mathbb{R}$ as 
\begin{align*}
q(x)= q_1 ^2 +p_1^2 + \dfrac{1}{N^2}\sum_{i=2}^\infty (q_i^2 + p_i^2), 
\end{align*}
where $p_i$, $q_i$ are the $i$-th coordinates of $x$ in the basis $\{e_i,f_i\}_{i \in \mathbb{N}}$ and $N$ is a natural number large enough so that 
\begin{align}
\label{hplambda}
\overline{\{ H < m \}} +B_\lambda  \subset  \{ q < 1\} \ \ \textrm{ and } \ \ \overline{\{ H < m \}} +B_\lambda \subset  \varphi(B_r),
\end{align}
for some $\lambda >0$.
The possibility of finding such values $N$ and $\lambda$ is a consequence of the discussion above about the set $\overline{\{ H < m \}}$, together with the observations that $\partial \varphi( \overline{B_r}) \subset \partial \overline{Z_1}$ and that $\varphi$ sends bounded sets into bounded sets.\\
Let us fix a real number $\mu$ such that
\begin{align*}
\pi < \mu < \min\{m,2 \pi \},
\end{align*}
and choose a smooth function $\rho:[0,+\infty[ \rightarrow \mathbb{R}$ satisfying
\begin{align*}
 \left\lbrace \begin{array}{l}
\rho(t)= m \ \ \ \textrm{ for any }t\leq 1\\ 
\rho(t) \geq t \mu \ \ \ \textrm{ for any } t\geq 0\\
\rho(t) = t \mu \ \ \ \textrm{ if } t \geq M \textrm{ large enough}\\
0 < \rho' (t) \leq \mu \ \ \ \textrm{ for any } t>1.
   \end{array} \right.
\end{align*}
We introduce the Hamiltonian 
\begin{align*}
K(x):=  \left\lbrace \begin{array}{l}
H(x) \ \ \ \textrm{ if } x \in \{ q < 1\} \\ 
\rho(q(x)) \textrm{ if } x \in \{ q \geq 1\}.\\
   \end{array} \right.
\end{align*} 
For any fixed $n \in \mathbb{N}$ we define the Hamiltonian $K_n : \mathbb{H}_n \rightarrow \mathbb{R}$ as
\begin{align*}
K_n:= K_{| \mathbb{H}_n}.
\end{align*}
Every $\mathbb{H}_n$ is finite dimensional, hence for fixed $n$ we can apply a standard minimax argument in order to find a critical point of $\mathcal{A}_{K_n}$ at positive critical level.\\
We do not reprove here this deep classical result, which is the main topic of Chapter 3 of \cite{HZ94} and is essentially equivalent to the Gromov's non-squeezing in $\mathbb{R}^{2n}$; nevertheless we give an outline of the proof of this fact, paying particular attention to the intermediate results we need to adapt to our setting.
This also gives us the opportunity to highlight what can go wrong in trying to extend the same proof to an infinite dimensional Hilbert symplectic space setting.\\
Let us describe the minimax setting. 
For $n\in \mathbb{N}$ fixed, we consider the Hilbert space of loops
\begin{align*}
E_n:=H^{\frac{1}{2}}(S^1,\mathbb{H}_n),
\end{align*}
which can be trivially identified with a subspace of $E=H^{\frac{1}{2}}(S^1,\mathbb{H})$. If we consider the splitting 
\begin{align*}
E_n=E_n ^- \oplus E_n ^0 \oplus E_n^+
\end{align*} 
with respect to the Fourier coefficients, then $E_n ^- \subset E^-$, $E_n ^0 \subset E^0$ and $E_n^+ \subset E^+$.
We introduce the subsets
\begin{align*}
\tilde{\Gamma}_\alpha &:= \{ x \in E^+  \ \big{|} \ | x|_{\frac{1}{2}} = \alpha \} \subset E^+,\\
\tilde{\Gamma}_\alpha^n &:= \{ x \in E_n^+  \ \big{|} \ | x|_{\frac{1}{2}} = \alpha \} \subset E_n ^+,
\end{align*}
and we translate them by $\varphi(0)$ to obtain
\begin{align*}
\Gamma_\alpha :=\tilde{\Gamma}_\alpha + \varphi(0), \\
\Gamma_\alpha^n :=\tilde{\Gamma}_\alpha^n + \varphi(0).
\end{align*}
The next statement is a consequence of the fact that $K$ vanishes in a neighbourhood of $\varphi(0)$.
\begin{lem}
For $\alpha >0$ small enough we have that 
\begin{align*}
\inf_{x \in \Gamma_\alpha^n} \mathcal{A}_{K_n}(x) \geq \inf_{x \in \Gamma_\alpha} \mathcal{A}_K(x)>0.
\end{align*}
\end{lem}
\proof
Since $E_n ^+ \subset E^+$ the left inequality is trivially true; moreover $K$ vanishes identically in a neighbourhood of $\varphi(0)$, hence Proposition \ref{diff} yields to
\begin{align*}
\mathcal{A}_K (\varphi(0))=0, \ \ \ d\mathcal{A}_K (\varphi(0))=0, \ \ \ d^2 \mathcal{A}_K (\varphi(0))[u,u]= |P^+ u |_{\frac{1}{2}} ^2 - | P^- u |_{\frac{1}{2}} ^2,
\end{align*}
for any $u \in E$, thus the claim follows by the Taylor formula with Peano's reminder.
\endproof
We define the subsets $\tilde{\Sigma}_{\tau} \subset E$ and $\tilde{\Sigma}_{\tau}^n \subset E_n$ as 
\begin{align*}
\tilde{\Sigma}_\tau &:= \{ x \in E \ \big{|} \ x = x^- + x^0 + s e^+, \ |x^- + x^0 |_{\frac{1}{2}} \leq \tau \ \textrm{ and } 0 \leq s \leq \tau \},\\
\tilde{\Sigma}_\tau^n &:= \{ x \in E_n \ \big{|} \ x = x^- + x^0 + s e^+, \ |x^- + x^0 |_{\frac{1}{2}} \leq \tau \ \textrm{ and } 0 \leq s \leq \tau \},
\end{align*}
where $\tau >0$ and 
\begin{align*}
e^+ (t) := \dfrac{e^{2 \pi tJ} e_1}{\sqrt{2 \pi}} 
\end{align*}
is a circle in $E^+_1 \subset E^+$.
We have $|e^+|_{\frac{1}{2}} ^2= 1$, $|e^+|_{L^2} ^2= \dfrac{1}{2 \pi}$ and we denote with $\partial \tilde{\Sigma}_{\tau}$ (resp. with $\partial \tilde{\Sigma}_{\tau}^n$) the boundary of $\Sigma_{\tau}$ in $E^- \times E^0 \times \mathbb{R}e^+$ (resp. of $\Sigma_{\tau}^n$ in $E^-_n \times E^0_n \times \mathbb{R}e^+$).
We denote the shift of $\tilde{\Sigma}_\tau$ and $\tilde{\Sigma}^n_\tau$ by $\varphi(0)$ with
\begin{align*}
\Sigma_\tau  :=\tilde{\Sigma}_\tau + \varphi(0),\\
\Sigma^n_\tau  :=\tilde{\Sigma}^n_\tau + \varphi(0).
\end{align*}
\begin{lem}
If $\tau$ is big enough then $\mathcal{A}_K \big{|}_{\partial \Sigma_\tau} \leq 0$ and $\mathcal{A}_{K_n} \big{|}_{\partial \Sigma_\tau^n} \leq 0$.
\end{lem}
\proof
To prove this we use the asymptotic behaviour of $K$. Since $a \big{|}_{E^- \times E^0}$ and $b_K$ are non positive, it follows that $\mathcal{A}_K \big{|}_{E^- \times E^0} \leq 0$.\\
Let us assume that $s=\tau$ or that $0\leq s \leq \tau$ and $| x^0 - \varphi(0) +  x^- |_{\frac{1}{2}} =\tau$. Since $K$ is the quadratic form $\mu q $ outside of a bounded set, we can find a constant $c\geq 0$ such that
\begin{align*}
K(x) \geq \mu q(x) - c, \ \ \forall x \in \mathbb{H}.
\end{align*}   
The splitting $E^- \times E^0  \times \mathbb{R} e^+$ is orthogonal with respect to the scalar product defined by the quadratic form $x \rightarrow \int_0 ^1 q(x(t))dt$, thus
\begin{align*}
\mathcal{A}_K (x) &= \frac{1}{2}s^2 - \frac{1}{2} | x^- |_{\frac{1}{2}}^2 - \int_0 ^1 K(x(t)) dt\\
& \leq \frac{1}{2}s^2 - \frac{1}{2} | x^- |^2 - \mu \int_0 ^1 q(x(t))dt + c \\
&=\frac{1}{2}s^2 - \frac{1}{2} | x^- |^2  - \mu \int_0 ^1 q(se^+) dt - \mu \int_0 ^1 q(x^0 - \varphi(0)) dt - \mu \int_0 ^1 q(x^-) dt + c\\
&=\frac{1}{2}s^2 - \frac{1}{2} | x^- |^2  - \frac{\mu s^2}{2 \pi}  - \mu \int_0 ^1 q(x^0 - \varphi(0)) dt - \mu \int_0 ^1 q(x^-) dt + c\\
&\leq - \frac{1}{2} (\frac{\mu}{\pi} -1)s^2 - \frac{1}{2} | x^- |^2 - \mu \int_0 ^1 q(x^0 - \varphi(0) ) dt - \mu \int_0 ^1 q(x^-) dt + c.
\end{align*}
Using the positivity of the quadratic form $q$ and the inequality $\pi < \mu$ we deduce that if $\tau$ is large enough then the quantity above is non positive as long as $s=\tau$ or $| x^0 - \varphi(0) + x^- |_{\frac{1}{2}} = \tau$. In particular $\mathcal{A}_{K_n} \big{|}_{\partial \Sigma_\tau^n} \leq 0$, because $\partial \Sigma_\tau^n \subset \Sigma_\tau$.
\endproof
Throughout the entire chapter we are assuming that $D \varphi$ is bounded, hence we can use Lemma \ref{lemlip} to deduce that the gradient equation 
\begin{align}
\label{qqq}
\dot{x}=-\nabla_{\frac{1}{2}} \mathcal{A}_K(x), \ \ \ x \in E
\end{align}
is globally Lipschitz continuous, thus it defines a unique global flow
\begin{align*}
\mathbb{R}\times E &\rightarrow E \\
 (t,x) &\mapsto \varphi_t(x)=: x \cdot t,
\end{align*}
which, as well as its inverse, maps bounded sets into bounded sets.\\
The same is clearly true also for the flow of
\begin{align}
\label{rrr}
\dot{x}=-\nabla_{\frac{1}{2}} \mathcal{A}_{K_n}(x)= x^- - x^+ + \nabla_{\frac{1}{2}} b_{K_n}(x), \ \ \ x \in E_n.
\end{align}
The vector field \eqref{rrr} is a compact perturbation of a linear one, indeed the compactness of 
\begin{align*}
E_n &\rightarrow E_n\\
x &\mapsto \nabla_{\frac{1}{2}} b_{K_n}(x)=T^*\nabla K_n(x)
\end{align*}
is a consequence of Sobolev's compact embedding theorem which affirms that the linear embedding $T:H^{\frac{1}{2}}(S^1,\mathbb{H}_n) \hookrightarrow L^2(S^1,\mathbb{H}_n)$ is a compact map.
Using the variation of constants method it is not hard to obtain the following representation formula for the flow of \eqref{rrr}.
\begin{lem}\upshape{{\cite{HZ94} (Chapter 3, Lemma 7)}} 
\label{fluss6}
The flow of $\dot{x}= -\nabla \mathcal{A}_{K_n}(x)$ admits the representation
\begin{align*}
x\cdot t &=  e^t x^- +x^0  +e^{-t} x^ + + l(t,x) 
\end{align*}
where $l: \mathbb{R} \times E_n \rightarrow E_n$ is a compact map.
\end{lem}
\begin{oss}
The flow of \eqref{qqq} cannot be seen as a compact perturbation of a linear flow because the Sobolev's embedding $T:H^{\frac{1}{2}}(S^1,\mathbb{H}) \hookrightarrow L^2(S^1,\mathbb{H})$ is not compact when $\mathbb{H}$ is infinite dimensional (cfr. Example \ref{esempio}).
\end{oss}
Having a representation formula for the flow of \eqref{rrr}, the main ingredient to prove the linking lemma below is the Leray-Schauder degree, which is defined for maps which are compact perturbations of the identity.  
\begin{lem}\upshape{{\cite{HZ94} (Chapter 3, Lemma 10)}} 
\label{link}
Let $\varphi_t$ be the flow of $\dot{x}=- \nabla_\frac{1}{2} \mathcal{A}_{K_n}(x)$, then for $\alpha>0$ small enough and $\tau>0$ big enough we have 
\begin{align*}
\varphi_t (\Sigma_\tau^n)\cap \Gamma_\alpha^n \neq \emptyset \ \ \textrm{ for any } t \geq 0.
\end{align*}
\end{lem}
\begin{oss}
By the previous remark, if $\mathbb{H}$ is infinite dimensional the gradient flow has no compactness property and this prevents the possibility of proving a linking lemma by means of the Leray-Schauder degree.
\end{oss}
We finally have all the ingredients to conclude the following.
\begin{lem}
\label{succpticr}
For any fixed $n \in \mathbb{N}$ there exists a critical point $x_n$ of $\mathcal{A}_{K_n}:E_n \rightarrow \mathbb{R}$ such that $x_n(t) \in \varphi(B_r) \cap \mathbb{H}_n$ for any $t \in [0,1]$ and
\begin{align*}
 0 < \inf_{x \in \Gamma_\alpha} \mathcal{A}_{K}(x_n) \leq \mathcal{A}_{K_n} (x_n) \leq \sup _{x \in \Sigma_{\tau}} \mathcal{A}_{K}(x_n)< +\infty.
\end{align*}
\end{lem}
\proof
For any fixed $n\in \mathbb{N}$ the equation \eqref{rrr} defines a global flow $\varphi_t$ on $E_n$ and the family of sets 
\begin{align*}
\mathcal{F}_n := \{ \varphi_t ( \Sigma_{\tau}^n) \ \big{|} \ t\geq 0\},
\end{align*}
is clearly positively invariant under the flow $\varphi_t$.\\
The value 
\begin{align*}
c(\mathcal{A}_{K_n},\mathcal{F}_n):= \inf_{t \geq 0} \sup _{x \in \varphi_t ( \Sigma_{\tau}^n)} \mathcal{A}_{K_n}(x)
\end{align*}
is finite, indeed $\mathcal{A}_{K_n} (\varphi_t(x))$ is a decreasing function of $t$ and $\mathcal{A}_{K}$ is bounded on bounded sets, thus
\begin{align*}
\inf_{t \geq 0} \sup _{x \in \varphi_t ( \Sigma_{\tau}^n)} \mathcal{A}_{K_n}(x) \leq \sup _{x \in \Sigma_{\tau}^n} \mathcal{A}_{K_n}(x) \leq \sup _{x \in \Sigma_{\tau}} \mathcal{A}_{K}(x)< +\infty.
\end{align*} 
Moreover Lemma \ref{link} implies that if $\tau$ is big enough then 
\begin{align*}
\sup _{x \in \varphi_t ( \Sigma_{\tau}^n)} \mathcal{A}_{K_n}(x) \geq  \inf_{x \in \Gamma_\alpha^n} \mathcal{A}_{K_n}(x) \geq  \inf_{x \in \Gamma_\alpha} \mathcal{A}_{K}(x)>0
\end{align*}
for any $t \geq0$, hence
\begin{align*}
c(\mathcal{A}_{K_n},\mathcal{F}_n) =  \inf_{t \geq 0} \sup _{x \in \varphi_t ( \Sigma_{\tau}^n)} \mathcal{A}_{K_n}(x) > 0.
\end{align*}
We can therefore apply the minimax lemma (Theorem \ref{m}) and, for any fixed $n$, we get a (PS)$^c$ sequence $\{x_n ^k\}_{k \in \mathbb{N}}\in H^{\frac{1}{2}}(S^1,\mathbb{H}_n)$. Since $\mathbb{H}_n$ is finite dimensional, one can show that the functional $\mathcal{A}_{K_n}$ satisfy the (PS) condition (see Lemma 6, Chapter 3 of \cite{HZ94}), therefore for any fixed $n$ we can find a critical point $x_n$ whose critical level is positive. Moreover it is not hard to see that any $x_n$ is supported in $\varphi(B_r) \cap \mathbb{H}_n$ (see Proposition 2, Chapter 3 of \cite{HZ94}).
\endproof
A critical point of $\mathcal{A}_{K_n}:E_n \rightarrow \mathbb{R}$ at level $c_n$ is a solution $x_n$ of
\begin{align*}
\left\lbrace \begin{array}{l}
\dot{x}_n = J\nabla K_n (x_n)\\
\mathcal{A}_{K_n}(x_n)=c_n,
\end{array} \right.
\end{align*} 
hence if we apply Lemma \ref{succpticr} we can find a sequence $\{x_n\}_{n \in \mathbb{N}}$ of loops supported in $\varphi(B_r)$ such that 
\begin{align}
\label{ffff}
\left\lbrace \begin{array}{l}
\dot{x}_n = J\nabla K_n (x_n)\\
 0 < \delta \leq \mathcal{A}_{K_n}(x_n)\leq \Delta < + \infty.
   \end{array} \right.
\end{align} 
It is now time to use the properties of approximation schemes in order to show that the existence of a sequence $\{x_n\}_{n \in \mathbb{N}}$ as in \eqref{ffff} implies the existence of a (PS)$^c$ sequence for $\mathcal{A}_F$ at positive level. 
\begin{prop}
Let $(\mathbb{H},\omega)$ be a symplectic Hilbert space endowed with a compatible inner product and $\{P_n\}_{n \in \mathbb{N}}$ the approximation scheme of orthogonal projections onto finite dimensional symplectic subspaces.
Let $\varphi \in Symp (\mathbb{H},\omega)$ be a symplectic diffeomorphism such that $D \varphi$, $D \varphi^{-1}$ are bounded and 
\begin{enumerate}[i)]
\item $(I-P_n) {\varphi^{-1}}_{\vert P_n \mathbb{H}} \underset{n \rightarrow + \infty} \longrightarrow 0$ uniformly on bounded sets,\\
\item $[P_n , D \varphi^{-1} (x)^*]  \underset{n \rightarrow + \infty} \longrightarrow 0$ in operators' norm, uniformly in $x \in \mathbb{H}$ on bounded sets.
\end{enumerate}
If $\{y_n\}_{n \in \mathbb{N}}$ is a sequence of critical points for $\mathcal{A}_{K_n}:E_n\rightarrow \mathbb{R}$ as in \eqref{ffff}, then the sequence $\{x_n\}_{n \in \mathbb{N}}$ of loops $x_n:=\varphi^{-1} (y_n)$ admits a \emph{(PS)}$^c$ subsequence at positive level for the Hamiltonian action functional $\mathcal{A}_{F}:E \rightarrow \mathbb{R}$.
 \label{propns}
\end{prop}
\proof
Let us set $\phi := \varphi^{-1}$.
For any $n \in \mathbb{N}$, if we consider the inclusion $i_n:P_n \mathbb{H} \rightarrow \mathbb{H}$ and the Hamiltonian $K_n:= K \circ i_n$, we get
\begin{align}
\nabla K_n (z)= (D i_n)^* \nabla K (z) = P_n \nabla K (z),
\end{align}
for any $z \in \mathbb{H}$. Therefore for any critical point $y_n$ of $\mathcal{A}_{K_n}$ it holds
\begin{align}
\label{qualc}
\dot{y}_n(t) = J\nabla K_n (y_n(t))= JP_n \nabla K (y_n(t)).
\end{align} 
We recall that 
\begin{align*}
K_{| \varphi(B_r)}= F \circ \phi: \varphi(B_r) \rightarrow \mathbb{R}
\end{align*}
with $F$ radial Hamiltonian and that the loops $y_n$ are supported in $\varphi(B_r) \cap P_n \mathbb{H}$.
Using \eqref{qualc} together with the fact that $\phi$ is symplectic, for the loops $x_n= \phi(y_n)$ we compute
\begin{align}
\label{abcde}
\begin{split}
\dot{x}_n(t) &= D \phi (y_n(t)) \dot{y}_n(t)=D \phi (y_n(t)) J P_n \nabla K(y_n(t))\\ &= J D \phi (y_n(t))^{-*} P_n D \phi (y_n(t))^* \nabla F(x_n(t))\\
&= J D \phi (y_n(t))^{-*} D \phi (y_n(t))^* P_n \nabla F(x_n(t)) +\\
& \ \ \ +J D \phi (y_n(t))^{-*} [P_n , D \phi (y_n(t))^*] \nabla F(x_n(t))\\
&=J P_n \nabla F(x_n(t)) + J D \phi (y_n(t))^{-*} [P_n , D \phi (y_n(t))^*] \nabla F(x_n(t)).
\end{split}
\end{align}
The gradient of $F:\mathbb{H} \rightarrow \mathbb{R}$ writes as
\begin{align*}
\nabla F(z)= 2 g' (\|z\|^2) z,
\end{align*}
thus we get
\begin{align*}
\nabla F(x_n(t)) &= 2 g' (\|x_n(t)\|^2) x_n(t) =  2 g' (\|x_n(t)\|^2) \phi(y_n(t))
\end{align*}
hence
\begin{align}
\label{lalala}
 P_n \nabla F(x_n(t))=\nabla F(x_n(t))-(I-P_n)2 g' (\|x_n(t)\|^2) \phi(y_n(t)),
\end{align}
for any $t\in [0,1]$.\\
By substituting \eqref{lalala} we can rewrite \eqref{abcde} as 
\begin{align*}
\dot{x}_n(t)&=  J \big(  \nabla F(x_n(t)) -  (I-P_n)2 g' (\|x_n(t)\|^2) \phi(y_n(t)) \big) -\\ 
& \ \ \ -J D \phi (y_n(t))^{-*}  [P_n , D \phi (y_n(t))^*] \nabla F(x_n(t)),
\end{align*}
which, since 
\begin{align*}
\nabla_{L^2} \mathcal{A}_F (x_n) = -\big(J\dot{x}_n + \nabla F(x_n)\big),
\end{align*}
is equivalent to
\begin{align*}
-\nabla_{L^2} \mathcal{A}_{F}(x_n) =(I-P_n)2 g' (\|x_n(t)\|^2) \phi(y_n(t)) +  D \phi (y_n(t))^{-*}  [P_n , D \phi (y_n(t))^*] \nabla F(x_n(t)).
\end{align*} 
Since $y_n(t) \in \varphi(B_r) \cap P_n \mathbb{H}$, using assumption i) we get
\begin{align*}
(I-P_n)2 g' (\|x_n(t)\|^2) \phi(y_n(t)) \underset{n \rightarrow + \infty} \longrightarrow 0 \textrm{ uniformly in } t,
\end{align*}
moreover, by assumption ii) for any $t \in [0,1]$ we have 
\begin{align*}
 [P_n , D \phi (y_n(t))^*]  \underset{n \rightarrow + \infty} \longrightarrow 0 \ \ \textrm{ in the operators' norm}.
\end{align*}
Thus applying the dominated convergence theorem we deduce
\begin{align*}
\nabla_{L^2} \mathcal{A}_{F}(x_n) \underset{n \rightarrow + \infty} \longrightarrow 0  \ \ \textrm{ in } L^2(S^1,\mathbb{H}).
\end{align*}
Finally, because of the inequality
\begin{align*}
|\nabla_{\frac{1}{2}} \mathcal{A}_{F}(x_n)|_{\frac{1}{2}} =|T^*\nabla_{L^2} \mathcal{A}_{F}(x_n)|_{\frac{1}{2}} \leq | \nabla_{L^2} \mathcal{A}_{F}(x_n)|_{L^2},
\end{align*}
we get
\begin{align*}
\nabla_{\frac{1}{2}} \mathcal{A}_{F}(x_n) \underset{n \rightarrow + \infty} \longrightarrow 0  \ \ \textrm{ in }  H^{\frac{1}{2}}(S^1,\mathbb{H}).
\end{align*}
To find the sought (PS) sequence for $\mathcal{A}_F$ we need a small last step.\\
By construction of the sequence $\{y_n \}_{n \in \mathbb{N}}$ (cfr. \eqref{ffff}), we have that
\begin{align*}
0 < \delta \leq \mathcal{A}_{K_n}(y_n) \leq \Delta < + \infty.
\end{align*}
with $\delta,\Delta \in \mathbb{R}$ independent on $n$.\\
The symplectic action of closed orbits is preserved under symplectomorphisms, hence 
\begin{align*}
0 < \delta \leq \mathcal{A}_{F}(x_n) = \mathcal{A}_{K_n}(y_n) \leq \Delta < + \infty.
\end{align*}
Thus we can find a subsequence $\{ x_{n_k}\}_{n_k \in \mathbb{N}}$ such that $\mathcal{A}_{F}(x_{n_k}) \underset{n \rightarrow + \infty} \longrightarrow c>0$, namely the (PS)$^c$ sequence we were looking for.
\endproof 
The non-squeezing theorem we just proved generalizes the one obtained in \cite{Kuk95a} which applies to the family of so called elementary symplectomorphisms.
We do not give the rather technical definition of elementary symplectomorphism, but we remark that any such symplectomorphism is a compact perturbation of a linear map, which satisfies the following properties.
\begin{lem}\upshape{{\cite{Kuk95a} (Lemma 3)}} 
\label{lem3}
Let $(\mathbb{H},\omega)$ be a Hilbert symplectic space endowed with a compatible inner product $\langle \cdot, \cdot \rangle$, $B_r$ the ball centred in $0$ with radius $r$ and $\varphi:B_r \rightarrow \mathbb{H}$ be an elementary symplectomorphism.
Then for any $\epsilon>0$ and $r < + \infty$ there exists a natural number $n$ such that 
\begin{align}
\label{simplcpt}
\varphi(x)=L (I + \varphi_\epsilon) (I + \varphi_n)(x)
\end{align}
for every $x \in B_r$, where  $L:\mathbb{H}\rightarrow \mathbb{H}$ is a direct sum of rotations in the symplectic planes spanned by $\{e_i,f_i\}$, while $I + \varphi_\epsilon: \mathbb{H}\rightarrow \mathbb{H}$ and $I + \varphi_n: \mathbb{H}\rightarrow \mathbb{H}$ are smooth symplectomorphisms such that 
\begin{align*}
\|\varphi_\epsilon (y)\| \leq \epsilon \ \ \textrm{ for any } \ y \in (I + \varphi_n)B_r,
\end{align*}
and 
\begin{align*}
\varphi_n (P_n x, (I-P_n)x) = (\varphi_n^0 (P_n x), (I-P_n)x).
\end{align*}
\end{lem}
\begin{lem}\upshape{{\cite{Kuk95a} (Lemma 6)}}\label{lem6}
Let $U\subset \mathbb{H}$ be a bounded open set and $\varphi:U \rightarrow \mathbb{H}$ be an elementary symplectomorphism such that $\varphi^{-1}:\varphi(U) \rightarrow \mathbb{H}$ is bounded, then $\varphi^{-1}$ is also an elementary symplectomorphism.
\end{lem} 
Using the two lemmata above and the infinite dimensional non-squeezing theorem, we deduce the following.
\begin{cor}[Kuksin's infinite dimensional non-squeezing]
Let $(\mathbb{H},\omega)$ be a Hilbert symplectic space endowed with a compatible inner product $\langle \cdot, \cdot \rangle$, $B_r$ the ball centred in $0$ with radius $r$ and $Z_R$ a cylinder whose basis lays on a symplectic plane and has symplectic area $\pi R^2$. Let $\varphi:B_r \rightarrow \mathbb{H}$ be an elementary symplectomorphism such that the differentials of $\varphi$ and $\varphi^{-1}$ are bounded up to the second order; if $\varphi(B_r) \subset Z_R$ then $r\leq R$.
\end{cor}
\proof
By assumption $\varphi$ is an elementary symplectomorphism, hence it can be represented as a rotation composed with a small perturbation of a finite dimensional map as in \eqref{simplcpt}. Since $P_n$ commutes with $I$, $L$ and $\varphi_n$ we obtain
\begin{align*}
(I-P_n) {\varphi}_{\vert P_n \mathbb{H}} \underset{n \rightarrow + \infty} \longrightarrow 0 \textrm{ uniformly on bounded sets}.
\end{align*}
By Lemma \ref{lem6} we know that also $\varphi^{-1}$ is an elementary symplectomorphism, therefore by an analogous argument we obtain \begin{align*}
(I-P_n) {\varphi^{-1}}_{\vert P_n \mathbb{H}} \underset{n \rightarrow + \infty} \longrightarrow 0 \textrm{ uniformly on bounded sets}.
\end{align*}
We compute 
\begin{align*}
D \varphi (x) =  L D\big{(}(I + \varphi_\epsilon) (I + \varphi_n)(x)\big{)} =L(I + D\varphi_\epsilon(y)) (I + D \varphi_n(x)) 
\end{align*}
where $y=(I + \varphi_n)(x)$, thus
\begin{align*}
D \varphi (x)^* =  (I + D \varphi_n(x)^*) (I + D\varphi_\epsilon(y)^*)  L^*.
\end{align*}
This implies 
\begin{align*}
P_n D \varphi (x)^* &= P_n (I + D \varphi_n(x)^*) (I + D\varphi_\epsilon(y)^*)  L^*\\
 &= P_n L^* + P_n D \varphi_n(x)^*L^* + P_n D\varphi_\epsilon(y)^* L^*+ P_n  D \varphi_n(x)^* D\varphi_\epsilon(y)^*L^*
\end{align*}
and 
\begin{align*}
 D \varphi (x)^*  P_n &= (I + D \varphi_n(x)^*) (I + D\varphi_\epsilon(y)^*)  L^* P_n\\
 &=  L^* P_n +  D \varphi_n(x)^* L^*P_n +  D\varphi_\epsilon(y)^* L^*P_n +   D \varphi_n(x)^* D\varphi_\epsilon(y)^* L^*P_n.
\end{align*}
Since $[P_n,L^*]=[P_n,L^{-1}]=0$ and $[P_n, D \varphi_n(x)^*L^*] =0$ we get
\begin{align}
\label{eqabove}
[P_n, D \varphi (x)^*] = [P_n, D\varphi_\epsilon(y)^*L^*]+ [P_n, D \varphi_n(x)^* D\varphi_\epsilon(y)^*L^*].
\end{align}
Since $D^2 \varphi$ is bounded, using the mean value theorem we deduce that\\ $\|D\varphi_\epsilon(y) \|_{\mathcal{L}(\mathbb{H},\mathbb{H})}\underset{n \rightarrow + \infty} \longrightarrow 0$. Using \eqref{eqabove} together with the fact that $D \varphi$, $P_n$ and $L$ are bounded we get
\begin{align*}
[P_n, D \varphi (x)^*]\underset{n \rightarrow + \infty} \longrightarrow 0 \textrm{ in operators' norm, uniformly in} \ x \in \mathbb{H} \textrm{ on bounded sets}.
\end{align*}
An analogous argument implies that 
\begin{align*}
[P_n, D \varphi^{-1} (x)^*]\underset{n \rightarrow + \infty} \longrightarrow 0 \textrm{ in operators' norm, uniformly in} \ x \in \mathbb{H} \textrm{ on bounded sets},
\end{align*}
thus $\varphi$ is an admissible symplectomorphism and we can apply the non-squeezing theorem.
\endproof
\begin{oss} 
The biggest class of symplectomorphisms (which is not a group) for which our proof of non-squeezing works is the one for which the assumptions of Proposition \ref{propns} are fulfilled, namely:
\begin{align*}
Symp_A (\mathbb{H},\omega,\langle \cdot, \cdot \rangle):=& \big{\{} \varphi \in Symp(\mathbb{H},\omega) \ \big{|} \ \ D \varphi \textrm{ and } D \varphi^{-1} \textrm{ are bounded;} \\
& (I-P_n) {\varphi^{-1}}_{\vert P_n \mathbb{H}} \underset{n \rightarrow + \infty} \longrightarrow 0 \textrm{ uniformly on bounded sets}; \\
& [P_n , D \varphi^{-1} (x)^*]  \underset{n \rightarrow + \infty} \longrightarrow 0 \textrm{ in operators' norm, uniformly in} \ x \in \mathbb{H}\\
&\textrm{on bounded sets} \big{\}}.
\end{align*}
\end{oss}

\section{Non-squeezing as a critical point theory problem}
\label{sec4}
Let us consider the set of symplectomorphisms
\begin{align*}
Symp_B (\mathbb{H},\omega):= &\big{\{} \varphi \in Symp(\mathbb{H},\omega) \ \big{|} \ \  \varphi \textrm{ and } \varphi^{-1} \textrm{ have bounded differentials}\\
&\ \ \textrm{up to the third order.} \big{\}}.
\end{align*}
Adopting the same notation as in the previous section, our aim is to prove the following.
\begin{prop}
Let $(\mathbb{H},\omega)$ be a Hilbert symplectic space endowed with a compatible inner product, $B_r$ the ball centred in $0$ with radius $r$ and $Z_1$ a cylinder whose basis lays on a symplectic plane and has symplectic area $\pi$. Let $\varphi \in Symp_B (\mathbb{H})$, if $\varphi(B_r) \subset Z_1$ and there exists a \emph{(PS)}$^c$ sequence $\{x_n\}_{n \in \mathbb{N}}$ for $\mathcal{A}_K$ at level $c>0$, then $r \leq 1$.
\label{akc>0}
\end{prop}
It is worth to recall that the Hamiltonian $K$ is constructed starting from the symplectomorphism $\varphi$.
\begin{oss}
Given an arbitrary $\varphi \in Symp_B (\mathbb{H})$, writing down the action functional $\mathcal{A}_K$ explicitly is easy, but finding a suitable (PS)$^c$ for $\mathcal{A}_K$ can be an extremely difficult task (if possible).
Anyway this task could be made solvable by requiring additional conditions on $\varphi$.
\end{oss}
Starting from now we focus on proving Proposition \ref{akc>0}.
Let us consider a map $\varphi \in Symp_B (\mathbb{H},\omega)$ and assume that $\varphi(B_r) \subset Z_1$, with $r>1$. Under this assumption we obtain the following lemmata.
\begin{lem}
Any sequence $\{x_n\}_{n \in \mathbb{N}} \in E$ such that $\nabla_{\frac{1}{2}} \mathcal{A}_K(x_n) \underset{n \rightarrow + \infty} \longrightarrow 0$ is $H^{\frac{1}{2}}$-bounded.
\label{lembound}
\end{lem}
\proof
Let us assume that $\{x_n\}_{n \in \mathbb{N}}$ in unbounded sequence and define the bounded sequence $\{y_n\}_{n \in \mathbb{N}}$ of elements
\begin{align*}
y_n:= \dfrac{x_n}{|x_n|_{\frac{1}{2}}}
\end{align*}
with unit norm.
Since 
\begin{align*}
x_n ^+ -x_n ^- - \nabla_{\frac{1}{2}}b_K (x_n)= \nabla_{\frac{1}{2}} \mathcal{A}_K(x_n) \underset{n \rightarrow + \infty} \longrightarrow 0
\end{align*}
we can multiply by $\dfrac{1}{|x_n|_{\frac{1}{2}}}$ and deduce that
\begin{align}
\label{ooo}
y_n ^+ -y_n ^- - \dfrac{\nabla_{\frac{1}{2}}b_K (x_n)}{|x_n|_{\frac{1}{2}}} \underset{n \rightarrow + \infty} \longrightarrow 0.
\end{align}
Let us consider the Hamiltonian defined as $I(z):= \mu q(z)$ for any $z \in \mathbb{H}$. Using the fact that there exists a constant $L$ such that $|\nabla K(z) - \nabla I(z)| \leq L$ for any $z \in \mathbb{H}$, we get that 
\begin{align*}
|\dfrac{\nabla_{\frac{1}{2}}b_K (x_n)}{|x_n|_{\frac{1}{2}}} - \nabla_{\frac{1}{2}} I(y_n)|_{\frac{1}{2}}&=|T^*\big(\dfrac{\nabla K (x_n)}{|x_n|_{\frac{1}{2}}} - \nabla I(y_n)\big)|_{\frac{1}{2}}  \leq |\dfrac{\nabla K (x_n)}{|x_n|_{\frac{1}{2}}} - \nabla I(y_n)|_{L^2} \\&= \dfrac{1}{|x_n|_{\frac{1}{2}}}|\nabla K(x_n) - \nabla I(x_n)|_{L^2}\\
&\leq  \dfrac{L}{|x_n|_{\frac{1}{2}}} \underset{n \rightarrow + \infty} \longrightarrow 0,
\end{align*}
hence 
\begin{align*}
\dfrac{\nabla_{\frac{1}{2}}b_K (x_n)}{|x_n|_{\frac{1}{2}}} - \nabla_{\frac{1}{2}} I(y_n) \underset{n \rightarrow + \infty} \longrightarrow 0.
\end{align*}
This in combination with \eqref{ooo} implies that 
\begin{align}
\label{fff}
y_n ^+ -y_n ^- - \nabla_{\frac{1}{2}} I(y_n)  \underset{n \rightarrow + \infty} \longrightarrow 0.
\end{align}
Since $I(z)= \mu \|z\|^2$ with $\pi < \mu < 2 \pi$, then \eqref{fff} is equivalent to the condition 
\begin{align*}
A y_n \underset{n \rightarrow + \infty} \longrightarrow 0,
\end{align*}
with $A$ invertible linear operator (cfr. Example \ref{esempio}). Such a condition is fulfilled if and only if $|y_n|_{\frac{1}{2}}  \underset{n \rightarrow + \infty} \longrightarrow 0$, but we already know that $|y_n|_{\frac{1}{2}}=1$, hence the initial assumption that $\{x_n\}_{n \in \mathbb{N}}$ is unbounded leads to a contradiction.
\endproof
\begin{lem}
There exists a real number $\eta>0$ such that, given any $H^1$-bounded \emph{(PS)}$^c$ sequence $\{x_n\}_{n \in \mathbb{N}}$ of $\mathcal{A}_K$ for which $\nabla_{L^2}\mathcal{A}_K(x_n)$ is $L^2$-infinitesimal, then $\{x_n\}_{n \in \mathbb{N}}$ admits a \emph{(PS)}$^c$ subsequence 
$\{x_{n_k}\}_{n \in \mathbb{N}}$ 
whose elements are either loops taking values in $\overline{\{ H < m \}} + B_\eta  \subset \varphi(B_r)$ or in $ \mathbb{H} \backslash \overline{\{ H < m \}}$.
\label{lemaut}
\end{lem}
\proof
According to \eqref{hplambda} we can find a value $\lambda>0$ such that 
\begin{align*}
\overline{\{ H < m \}} + B_\lambda \subset \varphi(B_r).
\end{align*}
We claim that, up to a subsequence, if $x_n(t_n^0) \in \mathbb{H} \backslash (\overline{\{ H < m \}} + B_\lambda)$ for some time $t_n^0$ then $x_n(t) \notin \overline{\{ H < m \}} $ for any $t \in[0,1]$ (and vice-versa).\\
Indeed, if not there are times $0\leq t_n^0 <t_n ^1 \leq1$ for which $x_n(t^0 _n) \in \mathbb{H} \backslash (\overline{\{ H < m \}} + B_\lambda)$, $x_n(t^1_n) \in  \overline{\{ H < m \}}$ and $x_n(t) \in (\overline{\{ H < m \}} + B_\lambda) \backslash \overline{\{ H < m \}}$ for $t \in [t_n ^0 ,t_n^1]$.
Since $x_n$ is absolutely continuous, we get
\begin{align*}
x_n(t^1_n)= x_n(t^0_n) + \int_{t^0_n} ^{t^1_n}  \dot{x}_n (s) ds. 
\end{align*}
The point $x_n(t^1_n)$ has distance at least $\lambda$ from $x_n(t^0_n)$, therefore
\begin{align*}
\lambda \leq |x_n(t^1_n)-x_n(t^0_n)|=| \int_{t^0_n} ^{t^1_n}  \dot{x}_n (s) ds | \leq  \int_{t^0_n} ^{t^1_n}  |\dot{x}_n (s)| ds.
\end{align*}
We know that
\begin{align}
\label{est2}
|J \dot{x}_n +  \nabla K(x_n)|_{L^2}  \underset{n \rightarrow + \infty} \longrightarrow 0 
\end{align}
and since $K$ is constant on $(\overline{\{ H < m \}} + B_\lambda) \backslash \overline{\{ H < m \}}$, for any $t_n^0 < s < t_n^1$ we have that
\begin{align*}
\nabla K (x_n(s))=0.
\end{align*}
Thus \eqref{est2} implies that
\begin{align*}
\int_{t^0_n} ^{t^1_n} |\dot{x}_n(t)|^2 dt = \int_{t^0_n} ^{t^1_n}  |J \dot{x}_n(t)|^2  dt  \underset{n \rightarrow + \infty} \longrightarrow 0,
\end{align*}
but this fact is in contradiction with the estimate
\begin{align*}
\lambda \leq \int_{t^0_n} ^{t^1_n}  |\dot{x}_n (s)| ds \leq  C \int_{t^0_n} ^{t^1_n}   |\dot{x}_n(t)|^2 dt,
\end{align*}
with $C$ independent on $n$, therefore $x_n$ cannot travel from $\mathbb{H} \backslash (\overline{\{ H < m \}} + B_\lambda)$ to $\overline{\{ H < m \}}$.
By setting $\eta := \frac{\lambda}{2}$ we deduce the result.
\endproof
\begin{lem}
Let $\{x_n\}_{n \in \mathbb{N}}$ be a $H^1$-bounded \emph{(PS)}$^c$ sequence for $\mathcal{A}_K$ such that $\nabla_{L^2}\mathcal{A}_K(x_n)$ is $L^2$-infinitesimal. If all the loops $x_n$ have support in the set $\mathbb{H} \backslash \overline{\{ H < m \}}$ then the level $c$ is not positive.
\label{PSK<0}
\end{lem}
\proof
For any loop $x$ taking values in $ \mathbb{H} \backslash \overline{\{ H < m \}}$ we have that
\begin{align*}
\nabla K(x)= 2 \rho '(q(x) Q x
\end{align*}
where $Q:= dq$ is the linear map which, on the Hilbert symplectic basis $ \{e_i,f_i\}_{i\in \mathbb{N}} $ of $\mathbb{H}$, is written as 
\begin{align*}
Q:\mathbb{H} &\rightarrow \mathbb{H},\\
(q_1,p_1)&\mapsto (q_1,p_1), \\
(q_i,p_i) &\mapsto  \dfrac{1}{N^2}(q_i,p_i) \textrm{ \ \ if \ \ } i\geq 2.
\end{align*}
We are interested in the relation between the Fourier coefficients of a curve $x \in L^2(S^1,\mathbb{H})$ and the ones of the curve $Qx$.\\
Let
\begin{align*}
x(t) = \sum_{k \in \mathbb{Z}} e^{2 \pi k Jt} x^k,
\end{align*}
we write any coefficient $x^k \in \mathbb{H}$ as
\begin{align*}
x^k=(x^k _1,\hat{x} ^k _1),
\end{align*} 
where $x_1 ^k$ is given by the components of $x^k$ in the 2-dimensional plane spanned by $\{ e_1,f_1\}$. We denote with $J_1$ the two dimensional linear operator mapping $e_1 \mapsto f_1$ and $f_1 \mapsto -e_1$ and we define $\hat{J}_1:=J-J_1$.\\
The image of $x$ via $Q$ is 
\begin{align*}
Q(x(t)) = Q(\sum_{k} e^{2 \pi k Jt} x^k)= \sum_{k} Q(e^{2 \pi k Jt} x^k).
\end{align*}
The operators $J_1$ and $\hat{J}_1$ commute, thus for any $k \in \mathbb{Z}$ we deduce that 
\begin{align*}
Q(e^{2 \pi k Jt} x^k)&=Q(e^{2 \pi kt (J_1 + \hat{J}_1)} x^k) =Q(e^{2 \pi k t J_1} e^{2 \pi k t \hat{J}_1}x^k )\\
&=Q(e^{2 \pi k t J_1}x^k_1 + e^{2 \pi k t \hat{J}_1}\hat{x}^k_1 )=e^{2 \pi k t J_1}x^k_1 + \frac{1}{N^2} e^{2 \pi k t \hat{J}_1}\hat{x}^k_1 \\
&=e^{2 \pi k t J_1}x^k_1 + e^{2 \pi k t \hat{J}_1}  \dfrac{1}{N^2} \hat{x}^k_1=e^{2 \pi k t J_1}e^{2 \pi k t \hat{J}_1} Q (x^k)\\
&=e^{2 \pi k Jt} Q(x^k),
\end{align*} 
hence the $k$-th Fourier coefficient of $Qx$ is $Qx^k$.\\
Let us consider a $H^1$-bounded (PS)$^c$ sequence of loops taking values in the set $\mathbb{H} \backslash \overline{\{ H < m \}}$; inferring as in the second part of the proof of Proposition \ref{lemPS} we obtain that
 \begin{align*}
K(x_n) \underset{n \rightarrow + \infty} \longrightarrow c_1 \in \mathbb{R} \cup \{\pm \infty\} \ \ \textrm{ uniformly in } t.
\end{align*}
Since the sequence is $L^2$-bounded, it has to converge to a constant $c_1\in \mathbb{R}$, thus we deduce the uniform convergence of
\begin{align*}
\rho(q(x_n(t))) \underset{n \rightarrow \infty}\longrightarrow  c_1 \in \mathbb{R}
\end{align*}
and hence of
\begin{align*}
\rho ' (q(x_n(t))) \underset{n \rightarrow \infty}\longrightarrow  d  \in \mathbb{R}.
\end{align*}
This implies that
\begin{align*}
\nabla K(x_n(t)) -2  d Q x _n(t)  \underset{n \rightarrow \infty}\longrightarrow  0 \ \ \textrm{ uniformly in } t,
\end{align*}
and hence
\begin{align*}
\nabla K(x_n) -2  d Q x _n  \underset{n \rightarrow \infty}\longrightarrow  0 \ \ \textrm{ in } L^2(S^1,\mathbb{H}).
\end{align*}
By Lemma \ref{lemgrad} it follows that
\begin{align*}
|\nabla_{\frac{1}{2}} b_K(x_n) -2  dT^* Q x _n|_{\frac{1}{2}} = |T^*(\nabla K(x_n) -2  dQ x _n)|_{\frac{1}{2}} \leq  |\nabla K(x_n) -2  dQ x _n|_{L^2},
\end{align*}
thus
\begin{align*}
\nabla_{\frac{1}{2}} b_K(x_n) -2  d T^*Qx _n \underset{n \rightarrow \infty}\longrightarrow  0 \ \ \textrm{ in } H^{\frac{1}{2}}(S^1,\mathbb{H}).
\end{align*}
In order for $x_n=(x_n ^- ,x_n ^0 ,x_n ^+)$ to define a (PS)$^{c}$ sequence for $\mathcal{A}_{K}$ it is necessary that $\nabla_{\frac{1}{2}} a(x_n) - \nabla_{\frac{1}{2}} b_{K}(x_n)  \underset{n \rightarrow \infty}{\longrightarrow} 0$ in the $H^{\frac{1}{2}}$-norm, 
thus it is necessary that
\begin{align}
\label{nnn}
x_n^+ - x_n ^- - 2  d T^*Qx _n = \nabla_{\frac{1}{2}} a(x_n) - 2  d T^*Qx _n   \underset{n \rightarrow \infty}{\longrightarrow} 0,  \ \ \textrm{ in } H^{\frac{1}{2}}(S^1,\mathbb{H}).
\end{align}
Because of Lemma \ref{lemgrad} we know that 
\begin{align*}
& 2 d T^*Qx _n^0 =   2 d Q x^0 _n,\\
&2 d T^*Qx _n^k = \dfrac{2 d}{2 \pi |k|} Q x_n ^k  \textrm{ for } k \neq 0.
\end{align*}
and since $0 \leq d < 2 \pi - \epsilon$ for some $\epsilon>0$, then 
\begin{align*}
\frac{\epsilon}{2\pi} |x_n^+ - x_n^1|_{\frac{1}{2}}  &  \leq |(x_n^+ - x_n^1)  - 2  d T^*Q(x_n^+ - x_n^1) ^+|_{\frac{1}{2}} ,  \\
|x_n^-|_{\frac{1}{2}}   &\leq |x_n^-  + 2  d T^*Qx _n ^-|_{\frac{1}{2}} .
\end{align*}
Therefore \eqref{nnn} is possible only if $|x_n ^+ - x_n^1|_\frac{1}{2} ^2  \underset{n \rightarrow \infty}\longrightarrow 0$ and $|x_n ^-|_\frac{1}{2} ^2   \underset{n \rightarrow \infty}\longrightarrow 0$.\\
We know that $0 \leq d< 2\pi$ and we shall consider two cases: when $d \neq \pi$ and when $d = \pi$.\\
If $d \neq \pi$ then, by Lemma \ref{lemgrad}, in order for $\{x_n\}_{n \in \mathbb{N}}$ to be a (PS)$^{c}$ sequence it is also necessary that $|x_n ^1|_\frac{1}{2} ^2  \underset{n \rightarrow \infty}\longrightarrow 0$, and since 
\begin{align*}
\mathcal{A}_K(x_n)= \dfrac{1}{2} |x^+_n|_{\frac{1}{2}}^2 - \dfrac{1}{2} | x^-_n|_{\frac{1}{2}}^2- b_K(x_n),
\end{align*}
we get that
\begin{align*}
\mathcal{A}_K(x_n) \underset{n \rightarrow \infty}\longrightarrow  c,
\end{align*}
with $ c \leq 0$.\\
If $d=\pi$, then $\{x_n\}_{n \in \mathbb{N}}$ approaches curves of the form 
\begin{align*}
\overline{x}(t)= e ^{2 \pi J t} (q_1,p_1,0,0,\ldots)
\end{align*}
with $\| (q_1,p_1,0,0,\ldots) \|= \pi + \xi$, where, by the definition of $\rho$, we have $\pi + \xi \leq m$.
In this case it follows that
\begin{align*}
\mathcal{A}_K (x_n) \underset{n \rightarrow \infty}\longrightarrow  \pi + \xi -m \leq 0.
\end{align*}
\endproof
We are finally ready to prove Proposition \ref{akc>0}.
\proof
Let us assume by contradiction that, given $r>1$, the ball $B_r$ can be symplectically squeezed into $Z_1$ by a symplectomorphism $\varphi \in Symp_B (\mathbb{H})$.\\
By Lemma \ref{lembound} the (PS)$^c$ sequence $\{x_n\}_{n \in \mathbb{N}}$ has to be bounded, moreover Proposition \ref{lemPS} and Lemma \ref{lemaut} tell us that it is possible to modify the sequence in such a way that it becomes $H^1$-bounded and its elements lay either entirely in a fixed domain $\overline{\{ H < m \}} + B_\eta  \subset \varphi(B_r)$ on which $\mathcal{A}_K$ and $\mathcal{A}_H$ coincide, or outside of $\overline{\{ H < m \}}$.\\
The latter possibility is excluded by Lemma \ref{PSK<0}, hence we can find a $H^1$-bounded (PS)$^c$ sequence $\{x_n\}_{n \in \mathbb{N}}$ at positive level for $\mathcal{A}_H$. Therefore according to Proposition \ref{PSinv} we can find a (PS)$^c$ sequence $\{\varphi^{-1}(x_n)\}_{n \in \mathbb{N}}$ for $\mathcal{A}_F$ at positive level, but this is in contradiction with Proposition \ref{propF}.
\endproof

\newpage
 \thispagestyle{empty}
\section*{\huge References}
\begin{bibliography}{A}
\bibitem[Abb01]{Abb01}  A. Abbondandolo,  \emph{Morse Theory for Hamiltonian Systems},  Chapman $\&$ Hall/CRC Research Notes in Mathematics \textbf{425}, (2001).
\bibitem[AM15]{AM15}  A. Abbondandolo and P. Majer, \emph{A non-squeezing theorem for symplectic images of the Hilbert ball}, Calc. Var. Partial Diff. Equ. \textbf{54}, 1469-1506 (2015).
\bibitem[Bou94a]{Bou94a}  J. Bourgain,  \emph{Periodic nonlinear Schr\"odinger equation and invariant measures}, Comm. Math. Phys.  \textbf{166}, no.1, 1-26 (1994).
\bibitem[Bou94b]{Bou94b}  J. Bourgain, \emph{Approximation of solutions of the cubic nonlinear Schr\"odinger equations by finite-dimensional equations and nonsqueezing properties},  Internat. Math. Res. Not. \textbf{2}, 79-90 (1994).
\bibitem[CM74]{CM74}  P. R. Chernoff and J. E. Marsden, \emph{Properties of Infinite dimensional Hamiltonian systems}, LNM \textbf{425}, Springer, Berlin (1974).
\bibitem[CKS+05]{CKS+05} J. Colliander, M. Keel, G. Staffilani, H. Takaoka, and T. Tao, \emph{Symplectic non-squeezing of the Korteweg-de Vries flow}, Acta Math. \textbf{195}, 197-252 (2005).
\bibitem[CKS+10]{CKS+10} J. Colliander, M. Keel, G. Staffilani, H. Takaoka, and T. Tao, \emph{Transfer of energy to high frequencies in the cubic defocusing nonlinear}\\ \emph{Schr\"odinger equation}, Invent. Math. \textbf{181}, 39-113 (2010).
\bibitem[CLL97]{CLL97}  K. C. Chang, J. Liu and M. Liu, \emph{Nontrivial periodic solutions for strong resonance Hamiltonian systems}, Ann. Inst. H. Poincare- Analyse Non-lin., \textbf{14}, 103-117 (1997).
\bibitem[Fri85]{Fri85}  L. Friedlander, \emph{An invariant measure for the equation $u_{tt} - u_{xx} + u^3= 0$}, Commun. Math. Phys. \textbf{98}, 1-16 (1985).
\bibitem[HZ90]{HZ90}  H. Hofer and E. Zehnder, \emph{A new capacity for symplectic manifolds},  Analysis, et Cetera, 405-427 (1990).
\bibitem[HZ94]{HZ94}  H. Hofer and E. Zehnder,  \emph{Symplectic invariants and Hamiltonian dynamics}, Birkh\"auser (1994).
\bibitem[Kuk95a]{Kuk95a}  S. B. Kuksin,  \emph{Infinite-dimensional symplectic capacities and a squeezing theorem for Hamiltonian PDE's}, Commun. Math. Phys. \textbf{167}, 531-552 (1995).
\bibitem[Kuk95b]{Kuk95b}  S. B. Kuksin,  \emph{On squeezing and flow of energy for nonlinear wave equations}, Geom. Funct. Anal. \textbf{5}, 668-701 (1995).
\bibitem[Rou10]{Rou10}  D. Roum\' egoux, \emph{A symplectic non-squeezing theorem for BBM equation},  Dyn. Part. Differ. Equ. \textbf{7}, 289-305 (2010).
\bibitem[Zak74]{Zak74}  V. E. Zakharov,  \emph{Hamiltonian formalism for waves in nonlinear media with a disperse phase}, Izv. Vysch. Uchebn. Zaved. Radiofiz. \textbf{17}, 431-453 (1974).
\end{bibliography} 

\newpage
 \thispagestyle{empty}

\end{document}